\documentclass[11pt]{amsart}
\usepackage[margin=1in]{geometry}
\usepackage[english]{babel}
\usepackage[utf8]{inputenc}
\usepackage{csquotes}
\usepackage[T1]{fontenc}
\usepackage{etoolbox}
\usepackage{enumerate}
\usepackage{enumitem}
\usepackage{tabularx}
\usepackage{graphicx}
\usepackage{amsmath}
\usepackage{amsfonts}
\usepackage{amssymb}
\usepackage{amsthm}
\usepackage{thmtools}
\usepackage{mathtools}
\usepackage{float}
\usepackage[style=alphabetic, maxbibnames=4]{biblatex}
\usepackage{xparse}
\usepackage[all]{xy}
\usepackage{mathrsfs}

\defbibheading{toplevel}[\refname]{\part*{#1}}

\usepackage[hypertexnames=false]{hyperref}
\hypersetup{
    colorlinks,
    citecolor=black,
    linkcolor=black,
    urlcolor=black
}

\newtheorem{thm}{Theorem}[section]
\newtheorem{lema}[thm]{Lemma}

\newtheorem{coro}[thm]{Corollary}

\declaretheorem[style=definition,name=Definition,sibling=thm]{Def}
\declaretheorem[style=definition,name=Example,sibling=thm]{ex}
\declaretheorem[style=definition,name=Remark,sibling=thm]{rmk}

\newtheorem{thmx}{Theorem}

\newcounter{appendixcount}
\newcounter{savedsection}

\newcommand{\Br}{\mathrm{Br}_n^+}
\newcommand{\C}{\mathbb{C}}

\newcommand{\Q}{\mathbb{Q}}
\newcommand{\Z}{\mathbb{Z}}

\newcommand{\Stk}{\mathrm{Stk}}

\newcommand{\Var}{\mathrm{Var}}

\newcommand{\Hom}[2]{\operatorname{Hom}(#1,#2)}

\NewDocumentCommand{\Gr}{O{k} O{n}}{\operatorname{Gr}(#1,#2)}

\newcommand{\Corr}[1][]{\mathrm{Corr}(#1\Var)}

\newcommand{\TwTVarG}{T\Var_G^{\text{tw}}}
\newcommand{\TwBVarG}{B\Var_G^{\text{tw}}}

\newcommand{\IsoTwBVarG}{\mathrm{Cl}(\TwBVarG)}

\newcommand{\PGL}{\operatorname{PGL}}
\newcommand{\GL}{\operatorname{GL}}

\newcommand{\SL}{\operatorname{SL}}

\newcommand{\Sp}{\operatorname{Sp}}

\renewcommand{\Br}{\operatorname{Br}_G^+}

\newcommand{\g}{\mathfrak{g}}
\renewcommand{\b}{\mathfrak{b}}
\renewcommand{\t}{\mathfrak{t}}
\renewcommand{\sl}{\mathfrak{sl}}

\newcommand{\img}{\operatorname{Img }}

\newcommand{\rk}{\operatorname{rk}}
\newcommand{\coker}{\operatorname{coker}}

\newcommand{\mcP}{\mathcal{P}}
\newcommand{\mcC}{\mathcal{C}}

\newcommand{\mcX}{\mathcal{X}}
\newcommand{\mcY}{\mathcal{Y}}
\newcommand{\mcM}{\mathcal{M}}
\newcommand{\mcF}{\mathcal{F}}

\newcommand{\mcU}{\mathcal{U}}
\newcommand{\mcE}{\mathcal{E}}
\newcommand{\mcI}{\mathcal{I}}

\newcommand{\mcT}{\mathcal{T}}
\newcommand{\mcQ}{\mathcal{Q}}
\newcommand{\mcZ}{\mathcal{Z}}
\newcommand{\mcW}{\mathcal{W}}

\NewDocumentCommand{\preBetti}{O{R} O{\mcC}}{X^{#2}(#1)}
\NewDocumentCommand{\betti}{O{\SL_n} O{R} O{\mcC}}{\mcM_{\mathrm{B}}^{#3}(#1(#2))}
\NewDocumentCommand{\abbreviatedBetti}{O{\SL_n}}{\mcM_{\mathrm{B}}^\mcC(#1)}

\NewDocumentCommand{\Dolbeault}{O{G} O{d}}{\mcM_{\text{Dol}}^{#2}(#1)}

\NewDocumentCommand{\Betti}{O{G} O{\mcC}}{\mcM^{#2}_{\mathrm{B}}(#1)}
\newcommand{\tBetti}[1][G]{\widetilde{\mcM^{\mcC}_{\mathrm{B}}(#1)}}

\title{Motivic Mirror Symmetry for Character Stacks}
\author{Lucas de Amorin}
\date{}

\begin{document}

\begin{abstract}
	We prove a motivic version of T. Hausel and M. Thaddeus' Topological Mirror Symmetry for character stacks associated with arbitrary semisimple groups and with generic semisimple monodromies. It is an analogue of F. Loeser and D. Wyss' result for Chow motives of moduli spaces of Higgs bundles. To this end, we generalize A. Mellit's cell decomposition to arbitrary connected reductive groups. We use it to describe the automorphism groups of all points of the associated character stacks. Then we show that the Weil pairing induces a duality between cells that interchanges automorphisms and connected components. This duality implies the result.
\end{abstract}

\maketitle

\setcounter{tocdepth}{1}
\tableofcontents

\part*{Introduction}

Let $G$ be a connected reductive algebraic group and $\Sigma$ be a compact Riemann surface of genus $g$ with $m$ punctures. The character stack $\abbreviatedBetti[G]$ is a moduli space of $G$-local systems on $\Sigma$ with prescribed monodromies $\mcC=(\mcC_1,\ldots, \mcC_m)$ around each puncture. Concretely, it is the quotient stack of
\[\{(x,y,z)\in G^g\times G^g\times \mcC_1\times\cdots\times\mcC_m: [x_1,y_1]\cdots [x_g,y_g]z_1\cdots z_m = 1\}\]
by the simultaneous conjugation action of $G/Z(G)$, where $Z(G)$ is the center of $G$. The group $Z(G)^{2g}$ acts on this space by left multiplication on $(x,y)$. Given a subgroup $F\subset Z(G)$, the quotient $[\abbreviatedBetti[G]/F^{2g}]$ can be seen as a twisted $G/F$-character stack $\abbreviatedBetti[G/F]$.

Assume that $G$ is simply connected and semisimple. Let ${}^L\tilde G$ be the universal cover of the group whose root datum is dual to that of $G$. We conjecture that there is an equality
\[ [\abbreviatedBetti[G/F]]_{st} = [\mcM_{\mathrm{B}}^{\mcC^\vee}({}^L\tilde G / (Z(G)/F)^*)]_{st}\]
in the Grothendieck ring $K_0(\Stk_\C)$ of (nice) stacks, where $\mcC^\vee$ is ``dual'' to $\mcC$ and $[-]_{st}$ means the (naive) motive of the inertia stack suitably corrected to incorporate fermionic shifts. We make a precise statement in Section \ref{sec:statement} for generic $\mcC$ and prove it in Section \ref{sec:duality}. More precisely, we work under the following running assumptions:
\begin{itemize}
	\item[-] the tuple $\mcC$ is generic so that $\abbreviatedBetti[G/F]$ is an orbifold,
	\item[-] each conjugacy class $\mcC_i$ is semisimple and \emph{nice} in the sense that the centralizer of any of its points is a Levi subgroup, and
	\item[-] $\mcC_m$ is regular.
\end{itemize}

\begin{thmx}[Theorem \ref{conjMotivic}]\label{key-result}
	Let $G$ be a connected, simply connected, and semisimple complex algebraic group. Let $F$ be a subgroup of its center $Z(G)$ and ${}^L\tilde{G}$ be the universal covering of its Langlands dual group. Then
	\[ [\abbreviatedBetti[G/F]]_{st}=[\mcM_{\mathrm{B}}^{\mcC^\vee}({}^L\tilde{G}/(Z(G)/F)^*)]_{st}\]
	for any $g\geq 0$ and generic tuples $\mcC$ and $\mcC^\vee$ of nice semisimple conjugacy classes of $G$ and ${}^L\tilde{G}$, respectively, that are dual to each other, with $\mcC_m$ regular.
\end{thmx}

There are at least two reasons to expect this result. First of all, in all known cases \cite[etc]{GP, LGP, GPHV}, the motives of character stacks and varieties are polynomial in the class of the affine line. If this is true in general for each connected component of their inertia stacks, then the previous equality is equivalent to an analogous one at the level of $E$-polynomials. The latter is nothing but the Betti Topological Mirror Symmetry \cite{HT-TMS} for arbitrary groups, which is known to hold for $\SL_n$ and one generic puncture \cite{LM}.

The second reason is based on what happens for Higgs bundles. The $P=W$ theorem \cite{MS-PW, HMMS-PW, MSY} relates the weight filtration on the cohomology of character varieties to another filtration, the perverse one, on the cohomology of moduli spaces of Higgs bundles. Under this identification, the Betti Topological Mirror Symmetry becomes a perverse version of the original Topological Mirror Symmetry \cite{HT-TMS}. The latter is known to hold in greater generality \cite{GWZ}. It also has stronger versions such as \cite{MS}, \cite{WL}, and \cite{HL}. The last two are equalities at the level of Chow and Voevodsky motives. Moreover, the perverse filtration as well as Maulik--Shen's techniques are known to be motivic \cite{MSY2}. This may suggest that $P=W$ holds beyond cohomology. If this is true, there should be an analogue of \cite{WL, HL} on the Betti side. Although the naive motive should not see filtrations, we believe that our conjecture is a good first step towards such an analogue.

The main ingredient for Theorem \ref{key-result} is a generalization of A. Mellit's cell decomposition for $\GL_n$ \cite{Mellit} to arbitrary reductive groups. More precisely, we show the following. Assume that $\mcC_m$ is regular. Let $B$ be a Borel subgroup of $G$, $U$ its unipotent radical, $T\subset B$ a maximal torus, and $C_m\in T\cap \mcC_m$. Then $\abbreviatedBetti[G] \simeq [\tilde R / (B/Z(G))]$, where
\[ \tilde R=\{(x,y,z)\in G^g\times G^g\times \mcC_1\times\cdots\times \mcC_{m-1}: [x_1,y_1]\cdots [x_g,y_g]z_1\cdots z_{m-1} C_m \in U\}\]
and $B$ acts by simultaneous conjugation. Let $\tBetti[G]:=[\tilde R/(T/Z(G))]$. There is a map $\tBetti[G]\to \abbreviatedBetti[G]$.

\begin{thmx}[Theorem \ref{thm:cell-decomposition} and Corollary \ref{cor:cells-pullback}]\label{thm:decomposition}
	Let $G$ be a connected and reductive complex algebraic group, $g\geq 0$ and $m\geq 1$ be integers, and $\mcC_1,\ldots,\mcC_m$ be a generic tuple of nice semisimple conjugacy classes of $G$. Assume that $\mcC_m$ is regular. Then $\abbreviatedBetti[G]$ has a cell decomposition such that for each cell its preimage in $\tBetti[G]$ is a finite quotient of $(\C^\times)^{d -2i}\times \C^{r+i}$ for some $i$, where $d=\dim \abbreviatedBetti[G]$ and $r=\frac{1}{2}(\dim G-\rk G)$.
\end{thmx}

For $G$ of type A and with derived group simply connected, the result is stronger: there are no non-trivial finite quotients, and $\tBetti[G]\to \abbreviatedBetti[G]$ is Zariski locally trivial with fiber $U$, as in \cite{Mellit}. However, both properties fail for non-simply laced types or without the simple connectedness hypothesis. We do not know what happens for types D and E.

The strategy to prove this theorem follows \cite{Mellit}. In the first part of this work, we introduce a generalization of certain varieties associated with braids appearing in \cite{Mellit} and \cite{STZ}. Moreover, we introduce a $B$-action on them that matches the  ``bundle'' structure $\tBetti[G]\to\abbreviatedBetti[G]$ under the cell decomposition. A major difference between \cite{Mellit} and this work is the existence of non-trivial stabilizers, which do not occur for $\GL_n$. Furthermore, the proof that they are finite in general is quite different. In \cite{Mellit}, certain surfaces are constructed to this end. We do not know an analogue for general reductive groups. Instead, we use a direct approach, studying the symmetries of certain tangent algebras, see Subsection \ref{subsec:stabs}. After understanding these generalized braid varieties, in the second part we carry out the proof of the cell decomposition. We explain its overall strategy in Section \ref{sec:strategy}.

These cell decompositions are indexed by combinatorial objects called walks. Let $W$ be the Weyl group of $G$ and $\Delta^+$ be a set of simple roots of $G$. Let $\mcP$ be the associated Bruhat graph; its set of vertices is $W$ and there is a directed edge between $v_1,v_2\in W$ if $\ell(v_1)<\ell(v_2)$ and $v_2 = s v_1$ for a symmetry $s$ associated with some root in $\Delta^+$. For type $A_2$, it is
\[ \xymatrix{
	& s_\alpha s_\beta s_\alpha = s_\beta s_\alpha s_\beta&\\
	s_\beta s_\alpha \ar[ur] & & s_\alpha s_\beta \ar[ul] \\
	s_\alpha \ar[u] & & s_\beta \ar[u]\\
	& e\ar[ul]\ar[ur] &
}\]
where $s_\alpha$ and $s_\beta$ are the simple reflections. At each vertex there is exactly one edge for each element of $\Delta^+$. Hence, we can think of a word with letters in $\Delta^+$, read from right to left, as a sequence of instructions in $\mcP$. Given such a sequence, a walk is a path on $\mcP$ such that at each step if the edge associated with the corresponding instruction goes up, we follow it, and if not, we choose between following it and going down or ignoring it and staying. For example, in type $A_2$, the possible walks associated with $\alpha \beta \beta \alpha$ and starting at $e$ are
\begin{align*}
	(e, s_\alpha, s_\beta s_\alpha  , s_\alpha, e),\ (e, s_\alpha, s_\beta s_\alpha  , s_\alpha, s_\alpha), \text{ and } (e, s_\alpha, s_\beta s_\alpha  , s_\beta s_\alpha , s_\alpha s_\beta s_\alpha).
\end{align*}

To each walk on $G$ there is an associated walk on ${}^L\tilde G$ given by identifying each simple root with its coroot and staying at the same steps. As a consequence, to each cell of $\abbreviatedBetti[G]$ there is an associated cell of $\mcM_{\mathrm{B}}^{\mcC^\vee}({}^L\tilde G)$. In the third part of this work, we study this duality. We show that automorphisms of points in a cell correspond to connected components of its dual cell. Using this, we reduce Theorem \ref{key-result} to a statement about each cell. Roughly speaking, it says that taking fixed points and multiplying by the fermionic shift compensate each other, see Theorem \ref{conjMotivic2} for a precise statement. We think of it as an analogue of \cite[Theorem 0.3]{MS} and \cite[Theorem 6.20(i)]{HL}. This can be proved combinatorially using the description of the cells in the second part.

\begin{thmx}[Theorem \ref{conjMotivic2}]
	Let $G$ be a connected and semisimple complex algebraic group, $g\geq 0$ and $m\geq 1$ be integers, and $\mcC_1,\ldots,\mcC_m$ be a generic tuple of nice semisimple conjugacy classes of $G$. Assume that $\mcC_m$ is regular. Let $C\subset \tilde R$ be the preimage of a cell of the decomposition of Theorem \ref{thm:decomposition} and $\gamma\in (T/Z(G))\times Z(G)^{2g}$ fixing some point of $C$. Then
	\begin{align*}
		\left[\bigcup_{u\in U} C^{u\gamma}\right] q^{F(\gamma,C)} &= [C]
	\end{align*}
	in $K_0(\Stk_\C)$, where $q$ is the class of the affine line and $F(\gamma,C)$ is the fermionic shift of $\gamma$ on $\abbreviatedBetti[G/Z(G)]$ at the points of $C$.
\end{thmx}

\subsection*{Acknowledgments} The author wishes to thank G. Arnone, A. Barreto, C. di Fiore, A. González-Prieto, M. Hablicsek, and M. Mereb for very useful discussions, and T. Hausel for some corrections.

\subsection*{AI disclosure} No LLMs were used for the first publicly available version of this manuscript. For the current version, the author used Claude Opus 5.5 to correct typos, grammar, notation inconsistencies, bibliography errors, and minor gaps in arguments, and to suggest small simplifications. The author has verified the correctness of the manuscript and takes full responsibility for its contents.

\part{Twisted \texorpdfstring{$B$}{B} varieties}\label{part-one}

\section{Preliminaries}

\subsection{Conventions and notations}

Throughout, $G$ denotes a connected reductive complex algebraic group and $\g$ its Lie algebra. Once and for all, we fix a Borel subgroup $B \subseteq G$ and a maximal torus $T \subseteq B$, with Lie algebras $\b$ and $\t$ respectively. We call $U$ the unipotent radical of $B$ and $W$ the Weyl group of $G$. We think of $W$ as the quotient $N_G(T)/T$ of the normalizer of $T$ in $G$ by $T$. We denote by $\mathscr{X}$ and $\mathscr{X}^\vee$ the sets of characters and one-parameter subgroups of $T$. We identify $\mathscr{X}^\vee$ with the lattice $\t(\Z)\subset\t$ via $\lambda\mapsto d\lambda(1)$. The sets of roots, positive roots and simple roots will be denoted $\Phi$, $\Phi^+$ and $\Delta^+$ respectively. In particular, $\Delta^+\subset \Phi^+\subset\mathscr{X}$. We also denote by $P$, $Q$, $P^\vee$ and $Q^\vee$ the weight, root, coweight and coroot lattices. The Weyl group $W$ acts on all of them by conjugation. To each root $\alpha$ there is an associated symmetry $s_\alpha\in W$, a coroot $\alpha^\vee\in \mathscr{X}^\vee$, and a subgroup $G_\alpha\subset G$, isomorphic to either $\SL_2$ or $\PGL_2$, whose Lie algebra is the $\sl_2$-subalgebra of $\g$ associated with $\alpha$. If $\alpha$ is positive, it also has an associated one-dimensional $T$-stable subgroup $U_\alpha\subset U$, whose Lie algebra has weight $\alpha$.

\subsection{A few results of J. Tits}\label{subsec: Tits}

In this subsection we recall some results of J. Tits that we will use throughout the paper; we refer the reader to \cite{Bourbaki} for more details. Recall that the Weyl group has a presentation
\[\langle s_\alpha, \alpha\in \Delta^+|s_\alpha^2, (s_\alpha s_\beta s_\alpha...)_{m_{\alpha,\beta}}=(s_\beta s_\alpha s_\beta...)_{m_{\alpha,\beta}}, \alpha\neq\beta\in \Delta^+\rangle\]
where $m_{\alpha,\beta}$ are certain elements in $\{2,3,4,6\}$ and $(sts...)_n$ means the word of length $n$ given by alternating $s$ and $t$ starting with $s$.

Let $s_{\alpha_l}\cdots s_{\alpha_1}$ be a word in $\{s_\alpha:\alpha\in \Delta^+\}$. It represents the element $w\in W$ obtained by multiplying its letters. We say that it is of minimal length or a minimal-length representation of $w$ if $l = \ell(w)$, the length of $w$, that is, the minimum number of simple reflections needed to write $w$.

There are two kinds of operations on these words:
\begin{itemize}
    \item type $M1$: replace in the word a block of consecutive letters of the form $(s_\alpha s_\beta s_\alpha...)_{m_{\alpha,\beta}}$ by $(s_\beta s_\alpha s_\beta...)_{m_{\alpha,\beta}}$.
    \item type $M2$: erase two consecutive letters $s_\alpha s_\alpha$
\end{itemize}

\begin{thm}[Tits's Word Theorem {\cite{Tits69,Matsumoto}}]
    Any word can be reduced to a minimal-length one by operations of type $M1$ and $M2$. In addition, given two minimal-length representatives of the same element of $W$, one can be obtained from the other by a sequence of operations of type $M1$.
\end{thm}

\begin{thm}[Exchange condition {\cite[Ch.~IV, \textsection1.5, Prop.~4]{Bourbaki}}]
    Let $w\in W$ and $\alpha\in \Delta^+$. Then $\ell(s_\alpha w)=\ell(w)\pm 1$ and, if $\ell(s_\alpha w) < \ell(w)$, there exists a minimal-length representation of $w$ that starts with $s_\alpha$.
\end{thm}

\begin{coro}\label{cor:braid-rel-are-minimal}
    For any distinct simple roots $\alpha$ and $\beta$,  $(s_\alpha s_\beta s_\alpha...)_{m_{\alpha,\beta}}$ is of minimal length.
\end{coro}
\begin{proof}
    Assume that the statement is not true and let $k< m_{\alpha,\beta}$ be the largest $k$ such that $(s_\alpha s_\beta s_\alpha...)_{k}$ and $(s_\beta s_\alpha s_\beta...)_{k}$ are of minimal length. There is no $M1$ operation that we can apply to $(s_\beta s_\alpha s_\beta...)_{k}$. By Tits's Word Theorem, it is the only minimal-length representation of the element it represents. Therefore, by the Exchange condition, it follows that $(s_\alpha s_\beta s_\alpha...)_{k+1}=s_\alpha (s_\beta s_\alpha s_\beta...)_{k}$ is of minimal length too. By symmetry, so is $(s_\beta s_\alpha s_\beta...)_{k+1}$, a contradiction.
\end{proof}

We will also use the following result of Tits. A pinning of $G$ is a family of  central isogenies $\iota_\alpha:\SL_2\to G_\alpha$, $\alpha\in\Delta^+$, such that $\iota_\alpha(\mathrm{diag}(t,t^{-1}))=\alpha^\vee(t)$ for any $t\in \C^\times$. Pinnings always exist.

\begin{thm}[\cite{Tits}]\label{thm:Tits-braid}
    Let $(\iota_\alpha)_{\alpha\in\Delta^+}$ be a pinning of $G$ and $n_\alpha=\iota_\alpha\left(\begin{smallmatrix}0&1\\-1&0\end{smallmatrix}\right)$. Then
    \[(n_\alpha n_\beta n_\alpha...)_{m_{\alpha,\beta}}=(n_\beta n_\alpha n_\beta...)_{m_{\alpha,\beta}}\]
    for any distinct $\alpha,\beta\in\Delta^+$.
\end{thm}

\begin{coro}\label{lem:braid-rel-choices}
	For each simple root $\alpha$, let $\hat s_\alpha \in s_\alpha T \cap G_\alpha$. Then
	\[ (\hat s_\alpha \hat s_\beta \hat s_\alpha...)_{m_{\alpha,\beta}}=(\hat s_\beta \hat s_\alpha \hat s_\beta...)_{m_{\alpha,\beta}} \]
	for any distinct $\alpha,\beta\in \Delta^+$.
\end{coro}
\begin{proof}
	Let $(\iota_\alpha)_{\alpha\in\Delta^+}$ be a pinning of $G$. Each $\hat s_\alpha$ is of the form $\iota_\alpha\left(\begin{smallmatrix}0&a\\-a^{-1}&0\end{smallmatrix}\right)$ for some $a\in\C^\times$. Composing $\iota_\alpha$ with the conjugation by $\left(\begin{smallmatrix}d&0\\0&d^{-1}\end{smallmatrix}\right)$, where $d^2=a$, we get another pinning with $\hat s_\alpha=\iota_\alpha\left(\begin{smallmatrix}0&1\\-1&0\end{smallmatrix}\right)$. The corollary follows from Theorem \ref{thm:Tits-braid}.
\end{proof}

\subsection{Positive braids} \label{subsec:key-construction}

Define the monoid of positive braids $\Br$ by the Artin--Tits presentation
\[ \langle s_\alpha, \alpha\in \Delta^+|(s_\alpha s_\beta s_\alpha...)_{m_{\alpha,\beta}}=(s_\beta s_\alpha s_\beta...)_{m_{\alpha,\beta}}, \alpha\neq\beta\in \Delta^+\rangle.\]
We call the relations \textit{the braid relations}. There is a natural (monoid) quotient map $\Br\to W$.

For any element of $W$ there is a unique lift to $\Br$ with the same length, number of letters, due to Tits's Word Theorem. In this way we can always think of $W$ as a subset of $\Br$. Now, given any function $f:\Delta^+\to M$ to a monoid $M$ that satisfies the braid relations, there is a well-defined monoid morphism $\Br\to M$ and, hence, a function $W\to M$. By abuse of notation, we call all of them $f$ as well. We will use this construction throughout the text.

\subsection{Varieties over $G$}\label{sec:var-over-G}

A variety over $G$ is a pair $(X,f)$, where $X$ is a complex algebraic variety and $f:X\to G$ is a morphism, called its structure morphism. A morphism from $(X,f)$ to $(Y,g)$ is a morphism $h:X\to Y$ such that $g\circ h=f$. We denote by $\Var_G$ the category of varieties over $G$ and we usually omit the structure morphism from the notation. For example, any locally closed subvariety of $G$, in particular any point, is a variety over $G$ via the inclusion. Given a variety $X$, we denote by $\underline{X}$ the variety over $G$ whose structure morphism is the constant map to the identity $e\in G$.

The convolution of $(X,f)$ and $(Y,g)$ is the variety over $G$
\[(X,f)*(Y,g):=(X\times Y, \mu\circ(f\times g)),\]
where $\mu:G\times G\to G$ is the multiplication. It is associative up to canonical isomorphism and the point $e$ is a unit, so $\Var_G$ is a monoidal category.

A decomposition of a variety $X$ over $G$ is a finite partition $X=\bigsqcup_i X_i$ into locally closed subvarieties, each of them seen as a variety over $G$ by restricting the structure morphism. We say that $X$ decomposes as a disjoint union of $Y_1,\ldots,Y_n$ over $G$ if it has a decomposition $X=\bigsqcup_{i=1}^n X_i$ with $X_i\simeq Y_i$ as varieties over $G$.

\subsection{Bruhat decomposition}\label{sec:bruhat-parametrizations}

There is a particular decomposition we will use throughout this work: the Bruhat decomposition,
\[G=\bigsqcup_{w\in W}BwB.\]
Here the Bruhat cell $BwB$ is the double coset $B\hat wB$ of any lift $\hat w\in wT$ of $w$ to $G$; it does not depend on the lift. For a detailed account, we refer the reader to \cite[Section 28]{Humphreys}.

Bruhat cells can be parametrized as follows. Given $w\in W$, let $\Phi_w=\Phi^+\cap w^{-1}(-\Phi^+)$ and let $U^-_w$ and $U^+_w$ be the unipotent subgroups of $G$ associated with $\Phi_w$ and its complement $\Phi^+\setminus\Phi_w$, respectively. We remark that $U=U^-_wU^+_w=U^+_wU^-_w$. Then, given $\hat w\in wT$, there are parametrizations
\[BwB\simeq U^-_{w^{-1}}\times T\times U\simeq U\times T\times U^-_w\]
both given by $(u,t,u')\mapsto u\hat wtu'$.

\subsection{Partial Bruhat cells}\label{subsec:partial-bruhat-cells}

Given $w\in W$ and a lift $\hat w\in wT$ to $G$, we call $\hat wU_w^-$ a partial Bruhat cell. In this subsection we study convolutions of these partial cells as varieties over $G$.

\begin{lema}\label{lem:going-up}
    Let $\alpha$ be a simple root and $\hat s\in s_\alpha T$. If $w\in W$ is such that $w^{-1}\cdot \alpha$ is positive and $\hat w\in wT$, then $\hat s U_{s_\alpha}^- * \hat w U_{w}^-\simeq \hat s\hat w U_{s_\alpha w}^-$ as varieties over $G$.
\end{lema}
\begin{proof}
    This is equivalent to the product map $U_{w^{-1}\cdot \alpha}\times U^-_{w}\to U^-_{s_\alpha w}$ being an isomorphism. This is well known \cite[Section 28.1]{Humphreys}.
\end{proof}

Recall that, given a root $\alpha$ of $G$, there is an associated one-parameter subgroup $\alpha^\vee:\C^\times\to T$. This induces a variety over $G$, which we call $T_\alpha$; as a variety it is simply $\C^\times$ and its structure morphism is $\alpha^\vee$. Note that its image is $G_\alpha \cap T$. Let $\tau_\alpha$ be $\alpha^\vee(-1)$.

\begin{lema}\label{lem:not-going-up}
    Let $\alpha$ be a simple root and $\hat s\in s_\alpha T \cap G_\alpha$. Then $\hat s U_{s_\alpha}^- * \hat s U_{s_\alpha}^-$ decomposes as a disjoint union of $\tau_\alpha U_{s_\alpha}^-$ and $E * \hat s U_{s_\alpha}^-$ over $G$, where $E\subset T_\alpha*U_{s_\alpha}^-$ is the graph of a morphism $T_\alpha\to U_{s_\alpha}^-$.
\end{lema}
\begin{proof}
    This follows as in \cite[Subsection 5.4]{Mellit}. We include a full proof for the reader's convenience. We start by studying $\hat s U_{s_\alpha}^- * \hat s U_{s_\alpha}^-$. This can be done on $G_\alpha$ and hence it suffices to do it on $\SL_2$. We choose $T$, $B$ and $U$ as the diagonal, upper triangular, and upper unitriangular matrices respectively. Up to conjugation by an element of $T$ we may assume that $\hat s=\left(\begin{smallmatrix}
        0 & 1 \\
        -1 & 0
    \end{smallmatrix}\right)$. In this case, $\hat s U_{s_\alpha}^- * \hat s U_{s_\alpha}^-$ is given by $f:\C^2\to \SL_2(\C)$,
    \[f(z_1,z_2) = \left(\begin{matrix}
        0 & 1 \\
        -1 & 0
    \end{matrix}\right)\left(\begin{matrix}
        1 & z_1 \\
        0 & 1
    \end{matrix}\right)\left(\begin{matrix}
        0 & 1 \\
        -1 & 0
    \end{matrix}\right)\left(\begin{matrix}
        1 & z_2 \\
        0 & 1
    \end{matrix}\right).\]
    On $z_1=0$ we get $\tau_\alpha U_{s_\alpha}^-$. Over the complement, we have
    \[f(z_1,z_2) = \left(\begin{array}{cc}
        -z_1^{-1} & 1 \\
        0 & -z_1
    \end{array}\right)\left(\begin{array}{cc}
        0 & 1 \\
        -1 & 0
    \end{array}\right)\left(\begin{array}{cc}
        1 & z_2-z_1^{-1} \\
        0 & 1
    \end{array}\right).\]

    Going back to $G$, write $\iota:\SL_2(\C)\to G_\alpha \subset G$ for the homomorphism we used in the last paragraph. Hence, we can apply the previous analysis to conclude that $\hat s U_{s_\alpha}^- * \hat s U_{s_\alpha}^-$ decomposes as a disjoint union of $ \tau_\alpha U_{s_\alpha}^-$ and $\C^\times \times \C$ with structure morphism
    \[f(z_1,z_2) = \iota\left(\begin{array}{cc}
        -z_1^{-1} & 1 \\
        0 & -z_1
    \end{array}\right)\hat s \iota\left(\begin{array}{cc}
        1 & z_2-z_1^{-1} \\
        0 & 1
    \end{array}\right),\]
    which is of the form claimed.
\end{proof}

\section{Twisted \texorpdfstring{$B$}{B}-varieties over \texorpdfstring{$G$}{G}}\label{sec:B-varieties}

\subsection{Twisted \texorpdfstring{$T$}{T}-varieties}

We start by recalling the notion of torus equivariance as in \cite[Definition 5.9]{Mellit}. For $w\in W$ and $t\in T$, we write ${}^wt:=wtw^{-1}$. A variety $f:X\to G$ over $G$ endowed with an action of $T$ on $X$ is a $w$-twisted $T$-variety over $G$, or $w$-equivariant, if
\[f(t\cdot x) = {}^{w}tf(x)t^{-1}\]
for any $t\in T$ and $x\in X$. A twisted $T$-variety over $G$ is a $w$-twisted $T$-variety over $G$ for some $w\in W$. We denote by $\TwTVarG$ the category of twisted $T$-varieties over $G$. Its morphisms are $T$-equivariant morphisms over $G$. It is monoidal under the convolution.

\begin{lema}
    Let $(X,f)$ and $(X',f')$ be two varieties over $G$. If $X$ is $w$-equivariant and $X'$ is $w'$-equivariant, then $X* X'$ is $ww'$-equivariant.
\end{lema}
\begin{proof}
    We define $t\cdot(x,x')=({}^{w'}t\cdot x,t\cdot x')$. Then
    \[^{ww'}tf(x)f'(x')t^{-1}= f({}^{w'}tx){}^{w'}tf'(x')t^{-1}=f({}^{w'}tx)f'(tx')=(ff')(t(x,x'))\]
    for any $t\in T$ and $(x,x')\in X\times X'$.
\end{proof}

\begin{ex}
    An $e$-twisted $T$-variety is the same thing as a $T$-variety $X$ together with a $T$-equivariant morphism $X\to G$ for the action by conjugation of $T$ on $G$. In particular, $U$, $U_w^-$, $w\in W$, and any root subgroup $U_\alpha$, $\alpha\in \Phi^+$, are naturally $e$-twisted $T$-varieties. Remark that, for a simple root $\alpha$, $t\in T$ acts on $U_\alpha = U_{s_\alpha}^-\simeq \C$ by multiplication by $\alpha(t)$.
\end{ex}

\begin{ex}
	Any $T$-variety $X$ can be endowed with an $e$-twisted $T$-variety structure by setting $f:X\to G$ to be constantly the identity. In particular, for a positive root $\alpha$, we write $\underline{U_\alpha}$ for $U_\alpha$ with this structure morphism.
\end{ex}

\begin{ex}
    For a root $\alpha$, $T_\alpha$ can be endowed with an $s_\alpha$-twisted equivariant structure. Remember that $T_\alpha$ is given by the one-parameter subgroup $\alpha^\vee: \C^\times \to T$. The action of $T$ is given by $t\cdot z = \alpha(t)^{-1} z$. This is indeed $s_\alpha$-equivariant as $\alpha^\vee(\alpha(t)^{-1})={}^{s_\alpha}t t^{-1}$
\end{ex}

\begin{ex}\label{ex:Tw}
    In a similar spirit to the previous example, $T$ can be endowed with a $w$-equivariant structure for any $w\in W$. The structure morphism is just the inclusion $T\subset G$, but the action is given by $t\cdot z = {}^{w}tt^{-1} z $. We call this twisted variety $T(w)$. Remark that, for any root $\alpha$, $T(s_\alpha)$ and $T_{\alpha}$ do not agree in general.
\end{ex}

\begin{ex}
    Let $w\in W$ and $\hat w\in wT$. The point $\hat w$ over $G$, with trivial $T$-action, is a $w$-twisted $T$-variety. Thus $\hat wU_w^-$ and $U\hat wU_w^-$ are $w$-equivariant.
\end{ex}

\begin{ex}
	Let $w\in W$, and $f:X\to G$ be a $w'$-twisted $T$-variety over $G$. Assume that the image of $f$ lies in $T$. We define ${}^wX$ as the $ww'w^{-1}$-twisted $T$-variety whose structure morphism is $x\mapsto {}^wf(x)$ and whose action is $t\star x= {}^{w^{-1}}t\cdot x$. For instance, ${}^wT_\alpha = T_{w\cdot \alpha}$ for a root $\alpha$.
\end{ex}

\begin{lema}\label{lem:commute-T}
	Let $X$ and $Y$ be twisted $T$-varieties. Assume that $X$ is $w$-twisted and let $f$ be the structure morphism of $Y$. If the image of $f$ lies inside $T$, then $X*Y\simeq {}^wY*X$ and $Y*X\simeq X*{}^{w^{-1}}Y$.
\end{lema}
\begin{proof}
	Isomorphisms are $(x,y)\mapsto (y, f(y)^{-1}\cdot x)$ and $(y,x)\mapsto ({}^{w^{-1}}f(y)\cdot x, y)$.
\end{proof}

\subsection{Twisted \texorpdfstring{$B$}{B}-varieties}

A $w$-twisted $B$-variety over $G$ is a variety $(X,f)$ over $G$ with a $B$-action such that its induced $T$-action is $w$-equivariant and there exists an algebraic morphism $\phi:U\times X\to U$, $(u,x)\mapsto \phi_x(u)$ such that
\[f(u\cdot x) = \phi_x(u)f(x)u^{-1}\]
for any $u\in U$ and $x\in X$. We call $\phi$ its intertwining function, and we say that it is trivial if $\phi_x(u)=u$ for any $x\in X$ and $u\in U$. A twisted $B$-variety over $G$ is a $w$-twisted $B$-variety over $G$ for some $w\in W$. We denote by $\TwBVarG$ the category of twisted $B$-varieties over $G$. Its morphisms are $B$-equivariant morphisms over $G$. As before, the convolution makes it a monoidal category.

\begin{lema}\label{lem:magic-lemma}
    Let $(X,f)$ be a $w$-twisted $B$-variety over $G$. Then its intertwining function $\phi$ satisfies
    \[{}^wt\cdot \phi_x(u)=\phi_{t\cdot x}(t\cdot u)\]
    and
    \[\phi_x(uu') = \phi_{u'\cdot x}(u)\phi_x(u')\]
    for any $x\in X$, $t\in T$, and $u,u'\in U$.
\end{lema}
\begin{proof}
    Indeed,
    \[{}^wt \phi_x(u) f(x) u^{-1} t^{-1} = f(t\cdot (u\cdot x)) = f((t\cdot u)\cdot (t\cdot x)) = \phi_{t\cdot x}(t\cdot u){}^wt f(x) t^{-1} (t\cdot u)^{-1}\]
    and
    \[\phi_x(uu')f(x)(uu')^{-1}=f((uu')\cdot x) = f(u\cdot (u'\cdot x)) = \phi_{u'\cdot x}(u)\phi_x(u')f(x)u'^{-1}u^{-1}.\]
\end{proof}

\begin{lema}\label{lem:conv-B}
    If $(X,f)$ is a $w$-twisted $B$-variety over $G$ and $(Y,f')$ is a $w'$-twisted $B$-variety over $G$, $X*Y$ is a $ww'$-twisted $B$-variety over $G$.
\end{lema}
\begin{proof}
    Let $\phi:U\times X\to U$ and $\phi':U\times Y\to U$ be their intertwining functions. We define
    \[(tu)\cdot(x,y)=(({}^{w'}t\phi'_{y}(u))\cdot x,(tu)\cdot y)\]
    for $t\in T$ and $u\in U$. Lemma \ref{lem:magic-lemma} applied to $\phi'$ implies it is a $B$-action. Indeed,
    \begin{align*}
        (t'u')\cdot((tu)\cdot(x,y)) &=(t'u')\cdot(({}^{w'}t\phi'_{y}(u))\cdot x,(tu)\cdot y)\\
        &= (({}^{w'}t'\phi'_{(tu)\cdot y}(u'){}^{w'}t\phi'_{y}(u))\cdot x,(t'u'tu)\cdot y)\\
        &= (({}^{w'}(t't)\phi'_{u\cdot y}(t^{-1}\cdot u')\phi'_{y}(u))\cdot x,(t't(t^{-1}\cdot u')u)\cdot y)\\
        &= (({}^{w'}(t't)\phi'_{y}((t^{-1}\cdot u')u))\cdot x,(t't(t^{-1}\cdot u')u)\cdot y)\\
        &= (t'u'tu)\cdot(x,y)
    \end{align*}
    for any $t,t'\in T$ and $u,u'\in U$. We know that the induced $T$-action is $ww'$-twisted. For the $U$-action we have
    \[\phi_x(\phi'_{y}(u))f(x)f'(y)u^{-1}= f(\phi'_{y}(u)\cdot x)\phi'_{y}(u) f'(y)u^{-1}=f(\phi'_{y}(u)\cdot x)f'(u\cdot y)=(ff')(u\cdot (x,y))\]
    for any $u\in U$ and $(x,y)\in X\times Y$. We note that $\phi_x(\phi'_{y}(u))$ has the desired $T$-equivariance property.
\end{proof}

\begin{ex}
    A $B$-variety $X$ together with a $B$-equivariant morphism $f:X\to G$, for the action of $B$ by conjugation on $G$, is an $e$-twisted $B$-variety with trivial intertwining function. In particular, $U$ and all Bruhat cells $BwB$ are twisted $B$-varieties.
\end{ex}

\begin{ex}
	As a special case of the previous example, any $B$-variety $X$ can be made a twisted $B$-variety over $G$ just by setting $f:X\to G$ to be constantly the identity of $G$.
\end{ex}

\begin{ex}\label{ex:image-in-B}
    Any twisted $T$-variety with image inside $B$ can be endowed with a twisted $B$-equivariant structure. Indeed, the action of $B$ is given by restriction of scalars through $B\to B/U\simeq T$ and the intertwining function is $\phi_x(u)=f(x)\cdot u$, where $f$ is the function to $G$.
\end{ex}

The example allows us to generalize Lemma \ref{lem:commute-T}.

\begin{lema}\label{lem:commute}
	Let $X$ be a $w$-twisted $B$-variety and $Y$ be a twisted $T$-variety with image inside $T$. Then $X*Y \simeq {}^wY*X$ and $Y*X\simeq X*{}^{w^{-1}}Y$.
\end{lema}
\begin{proof}
	The same isomorphisms work.
\end{proof}

\begin{ex}
    Let $w\in W$ and $\hat w\in wT$. Then $\hat w U_w^-$ is a $w$-twisted $B$-variety. We already described how $T$ acts. It remains to explain the action of $U$. Let $u\in U$ and $\hat w z\in \hat wU_w^-$. Decompose $zu^{-1}$ as $u^+\cdot u^-$ with $u^+\in U_w^+$ and $u^-\in U_w^-$. Then set $u\cdot (\hat w z):=\hat w u^-$ and $\phi_{\hat w z}(u)=\hat w (u^+)^{-1}\hat w^{-1}$. It is simple to check that this defines a $w$-twisted $B$-variety. It is important to note that Lemma \ref{lem:going-up} is $B$-equivariant.
\end{ex}

\begin{lema}\label{lem:going-up-B}
    Let $\alpha$ be a simple root and $\hat s\in s_\alpha T$. If $w\in W$ is such that $w^{-1}\cdot \alpha$ is positive and $\hat w\in wT$, then $\hat s U_{s_\alpha}^- * \hat w U_{w}^-\simeq \hat s\hat w U_{s_\alpha w}^-$ as twisted $B$-varieties over $G$.
\end{lema}
\begin{proof}
    The important thing to note is that the decomposition $u^+u^-$ can be computed inductively, splitting at each step with respect to a simple root. More precisely, write, for each $w\in W$, $(-)_w^+:U\to U_w^+,(-)_w^-:U\to U_w^-$ for the unique functions such that $u=(u)_w^+(u)_w^-$ for any $u\in U$. Then, in the situation of the lemma,
    \begin{align*}
        (u)_{s_\alpha w}^- = \hat w^{-1}\left(\hat w(u)_w^+\hat w^{-1}\right)_{s_\alpha}^-\hat w\, (u)_w^-
    \end{align*}
    and
    \begin{align*}
        (u)_{s_\alpha w}^+ = \hat w^{-1}\left(\hat w(u)_w^+\hat w^{-1}\right)_{s_\alpha}^+\hat w
    \end{align*}
    for any $u\in U$. We leave it to the reader to check that these equalities imply the desired $B$-equivariance.
\end{proof}

\begin{ex}
    Similarly to the previous example, we can endow $U_w^-$, $w\in W$, with an $e$-twisted $B$-equivariant structure. Indeed, $T$ acts by conjugation and $U$ by $u\cdot z=(zu^{-1})_w^-$. The intertwining function is $\phi_z(u)=((zu^{-1})^+_w)^{-1}$.
\end{ex}

\subsection{The untwisting functor}

One could be tempted to define a twisted $B$-variety structure on a Bruhat cell $BwB$ by thinking of it as a convolution $U * T(w^{-1}) * \hat wU_w^-$. However, the induced $U$-action on $BwB$ is not the natural action by conjugation. Instead, we need to use what we call the untwisting functor.

Let $f:X\to G$ be an $e$-twisted $B$-variety over $G$ with intertwining function $\phi$. We define its untwist $\mcU(X)$ as $U*X$ as a variety over $G$ but with $B$-action given by
\[(tu')\cdot (u,x) = (tu'u (t\phi_x(u'))^{-1},(tu')\cdot x)\]
for $t\in T$, $u,u'\in U$ and $x\in X$.

\begin{lema}
    For any $e$-twisted $B$-variety $X$ over $G$, its untwist $\mcU(X)$ is a $B$-variety over $G$.
\end{lema}
\begin{proof}
    It is indeed an action because
    \begin{align*}
        t'\phi_{(tu')\cdot x}(u'')t\phi_x(u') &= t't\phi_{u'\cdot x}(t^{-1}\cdot u'')\phi_x(u') = t't\phi_{x}((t^{-1}\cdot u'')u')
    \end{align*}
    for any $t,t'\in T$, $u,u',u''\in U$, and $x\in X$, where both equalities follow from Lemma \ref{lem:magic-lemma}. The other properties can be checked in a similar way.
\end{proof}

This endows any Bruhat cell $BwB$ with a twisted $B$-variety structure as $\mcU(T(w^{-1})*\hat wU_w^-)$, whose $B$-action is literally given by conjugation.

\begin{lema}
    Let $w\in W$ and $\hat w\in wT$. Then the Bruhat cell $BwB$ is isomorphic as a twisted $B$-variety to $\mcU(T(w^{-1}) * \hat wU_w^-)$.
\end{lema}
\begin{proof}
    The multiplication $U\times T(w^{-1}) \times \hat wU_w^-\to BwB$ is bijective, over $G$, and equivariant.
\end{proof}

We collect a few useful properties of the untwist.

\begin{lema}\label{lem:magic-B-equivariancy}
    Let $X$ and $Y$ be twisted $B$-varieties. Then:
    \begin{enumerate}
        \item If $X$ is $e$-twisted, the projection onto the second coordinate $\pi_X:\mcU(X)\to X$ is $B$-equivariant. More generally, the morphism $Y * \mcU(X)\to Y * X$, $(y,u,x)\mapsto (u^{-1}\cdot y, x)$ is $B$-equivariant.
        \item If $X$ and $Y$ are $e$-twisted and $X$ has trivial intertwining function, a morphism $f:X\to \mcU(Y)$ over $G$ is $B$-equivariant if and only if $\pi_Y\circ f: X\to Y$ is $B$-equivariant.
        \item If $X$ and $Y$ are $e$-twisted, then $\mcU(X) * \mcU(Y) \to \mcU(X*Y)$, $(u_1,x,u_2,y)\mapsto (u_1\phi_x(u_2^{-1})^{-1},u_2^{-1}\cdot x, y)$, is a morphism of twisted $B$-varieties.
        \item If $X$ and $Y$ are $e$-twisted and $Y$ has trivial intertwining function, then $\mcU(X*Y)=\mcU(X)*Y$ as twisted $B$-varieties.
    \end{enumerate}
\end{lema}
\begin{proof}
    The first item is a direct computation. The ``only if'' direction of the second item follows from the previous item. For the ``if'' direction, write $f(x)=(u_1,y)$ and $f(u\cdot x)=(u_2,u\cdot y)$ for $x\in X$ and $u\in U$, with $u_1,u_2\in U$ and $y\in Y$. As $f$ is defined over $G$, calling $\psi_X:X\to G$ and $\psi_Y:Y\to G$ the structure morphisms, we must have $\psi_X(x) = u_1 \psi_Y(y)$ and $\psi_X(u\cdot x) = u_2 \psi_Y(u\cdot y)$. Now, $\psi_X(u\cdot x) = u \psi_X(x) u^{-1}$ and $\psi_Y(u\cdot y) = \phi_y(u) \psi_Y(y) u^{-1}$. Putting everything together, we get $u_2 = u u_1 \phi_y(u)^{-1}$ as desired.

    For the third item, we start by noticing that the morphism is defined over $G$. Hence, it suffices to check that the induced map to $X*Y$ is $B$-equivariant. Now, it is the composition of $\mcU(X)*\mcU(Y)\to X*\mcU(Y)$, $(u_1,x,u_2,y)\mapsto (x,u_2,y)$, and $X*\mcU(Y)\to X*Y$, $(x,u_2,y)\mapsto (u_2^{-1}\cdot x,y)$. Both maps are $B$-equivariant by the previous items.

    The last item follows from the definitions, as the intertwining function of $X*Y$ is $(u,(x,y))\mapsto \phi_x(u)$.
\end{proof}

\subsection{\texorpdfstring{$U$}{U}-morphisms}

The following notion is useful to work with untwists. Let $f_X:X\to G$ and $f_Y:Y\to G$ be twisted $B$-varieties (or $T$-varieties). A $U$-morphism $F:X\to Y$ is an equivariant morphism between the underlying $B$-varieties ($T$-varieties) such that $\eta(x):= f_X(x) f_Y(F(x))^{-1} \in U$ for any $x\in X$. We write $X\simeq_U Y$, and say that $X$ is $U$-isomorphic to $Y$ and that $F$ is a $U$-isomorphism, if $F$ is moreover an isomorphism between the underlying $B$-varieties ($T$-varieties). It is simple to check that compositions of $U$-morphisms are again $U$-morphisms.

\begin{lema}\label{lem:U-morphisms}
	Let $X$ and $Y$ be $e$-twisted $B$-varieties and $F:X\to Y$ be a $U$-morphism between them (or between their underlying twisted $T$-varieties). Then $\tilde F:\mcU(X)\to \mcU(Y)$, $(u,x)\mapsto (u\eta(x), F(x))$ is a morphism of $B$-varieties (resp. $T$-varieties) over $G$. Moreover, it is an isomorphism if and only if $F$ is.
\end{lema}
\begin{proof}
	It is defined over $G$, as $u\eta(x)f_Y(F(x))=uf_X(x)$. Its composition with $\pi_Y$ is $F\circ\pi_X$, which is $B$-equivariant by Lemma \ref{lem:magic-B-equivariancy}(1). Hence $\tilde F$ is $B$-equivariant by Lemma \ref{lem:magic-B-equivariancy}(2). If $F$ is an isomorphism, $(u,y)\mapsto(u\eta(F^{-1}(y))^{-1},F^{-1}(y))$ is the inverse of $\tilde F$. Conversely, if $\tilde F$ is an isomorphism, $y\mapsto\pi_X(\tilde F^{-1}(e,y))$ is the inverse of $F$, as $\tilde F(\eta(x)^{-1},x)=(e,F(x))$.
\end{proof}

\begin{lema}\label{lem:U-conv}
	Let $X$ and $Y$ be twisted $B$-varieties ($T$-varieties) and $Z$ be a twisted $B$-variety. If $F:X\to Y$ is a $U$-morphism, $X*Z\to Y*Z$, $(x,z)\mapsto (F(x),z)$, and $Z*X\to Z*Y$, $(z,x)\mapsto (\eta(x)^{-1}\cdot z,F(x))$, are $U$-morphisms.
\end{lema}
\begin{proof}
	It is straightforward to check.
\end{proof}

To finish this subsection, let us show that the isomorphisms of Lemma \ref{lem:not-going-up} are $U$-isomorphisms in the following sense. We say that a twisted $B$-variety $X$ decomposes ($U$-decomposes) as a disjoint union of twisted $B$-varieties $Y_1,\ldots,Y_n$ if it has a decomposition $X=\bigsqcup_i X_i$ over $G$ into $B$-stable pieces with $X_i\simeq Y_i$ ($X_i\simeq_U Y_i$).

\begin{lema}\label{lem:not-going-up2}
	Let $\alpha$ be a simple root and $\hat s\in s_\alpha T \cap G_\alpha$. Then $\hat s U_{s_\alpha}^- * \hat s U_{s_\alpha}^-$ $U$-decomposes as a disjoint union of $\tau_\alpha * U_{s_\alpha}^-$ and $T_\alpha * \hat s U_{s_\alpha}^-$.
\end{lema}
\begin{proof}
	We endow the variety $E$ of Lemma \ref{lem:not-going-up} with the $B$-structure of Example \ref{ex:image-in-B}. Then $E\to T_\alpha$ is a $U$-isomorphism, so $E*\hat sU^-_{s_\alpha}\simeq_U T_\alpha*\hat sU^-_{s_\alpha}$ by Lemma \ref{lem:U-conv}. We work out the details of the $U$-action and leave the rest to the reader.  Following the notation of the proof of Lemma \ref{lem:not-going-up}, it suffices to describe how $U$ acts on $(z_1,z_2)\in \C^2$. Remark that $U_{s_\alpha}^+$ is normal in $U$  \cite[section 17.4]{Humphreys}. Hence, $(-)_{s_\alpha}^-: U\to U_{s_\alpha}^-$ is a group homomorphism. Let $g:U_{s_\alpha}^-\to \C$ be the inverse of $z\mapsto \iota\left(\begin{smallmatrix}
		1 & z \\ 0 & 1
	\end{smallmatrix}\right)$. Then the $U$-action on $\C\simeq \hat s U_{s_\alpha}^-$ is given by $u\cdot z = z - g \left((u)_{s_\alpha}^- \right)$. Thus, the action on $(z_1,z_2)\in \C^2$ is $u\cdot (z_1, z_2) = (z_1, z_2 - g \left((u)_{s_\alpha}^- \right))$. It is clear from this description that the isomorphisms of Lemma \ref{lem:not-going-up} are $U$-equivariant.
\end{proof}

\begin{coro}
	Let $\alpha$ be a simple root and $\hat s\in s_\alpha T \cap G_\alpha$. Then $\mcU(\hat s U_{s_\alpha}^- * \hat s U_{s_\alpha}^-)$ decomposes as a disjoint union of $\mcU(\tau_\alpha * U_{s_\alpha}^-)$ and $\mcU(T_\alpha * \hat s U_{s_\alpha}^-)$ as twisted $B$-varieties.
\end{coro}

\subsection{Weak decompositions}

Let $X$ be a twisted $B$-variety and $Y_1,\ldots, Y_n$ be twisted $T$-varieties. We say that $X$ weakly decomposes as $Y_1,\ldots,Y_n$ if it has a decomposition $X=\bigsqcup_i X_i$ over $G$ into $B$-stable pieces with $X_i\simeq Y_i$ as twisted $T$-varieties. Similarly, we say that it weakly $U$-decomposes as $Y_1,\ldots, Y_n$ if instead $X_i\simeq_U Y_i$ as twisted $T$-varieties.

\begin{lema}
	Let $X, Y_1, \ldots, Y_n$ be $e$-twisted $B$-varieties. If $X$ weakly $U$-decomposes as $Y_1,\ldots, Y_n$, $\mcU(X)$ weakly decomposes as $\mcU(Y_1), \ldots, \mcU(Y_n)$.
\end{lema}
\begin{proof}
	This follows from Lemma \ref{lem:U-morphisms}.
\end{proof}

\section{Generalized braid varieties} \label{sec:unipotent-varieties}

In \cite{Mellit} and \cite{STZ} a central role is played by the varieties $f_{i}:\C\to \GL_n(\C)$, $z\mapsto P_i(\operatorname{Id}_n + zE_{i, i+1})$, where $P_i$ is the matrix that represents the transposition $(i\ i+1)$, and their convolutions. In this section we introduce analogues for any reductive group and study a few basic properties.

\subsection{Definition}

Choose once and for all a function $\sigma:\Delta^+\to G$ with $\sigma(\alpha)\in s_\alpha T\cap G_\alpha$. By Corollary \ref{lem:braid-rel-choices} and Subsection \ref{subsec:key-construction}, it induces a monoid morphism $\Br\to G$ and a function $W\to G$, which we will also denote $\sigma$. Let $\IsoTwBVarG$ be the monoid of isomorphism classes of $\TwBVarG$ with product induced by convolution.

\begin{thm}\label{thm:braid-var}
    Let $G$ be a reductive group and $\sigma:\Delta^+\to G$ be as above. Then, there exists a monoid morphism $\rho:\Br\to \IsoTwBVarG$ such that, for any $w\in W\subset \Br$, $\rho(w)$ is $\sigma(w) U_{w}^{-}$.
\end{thm}
\begin{proof}
    By Corollary \ref{cor:braid-rel-are-minimal} and Lemma \ref{lem:going-up-B}, it follows that
    $\alpha\in\Delta^+\mapsto \sigma(\alpha)U_{s_\alpha}^-$ satisfies the braid relations. Hence, it defines a morphism $\rho:\Br\to \IsoTwBVarG$. Furthermore, by Lemma \ref{lem:going-up-B}, $\rho(w)$ is $\sigma(w) U_{w}^{-}$ for any $w\in W\subset \Br$.
\end{proof}

\begin{rmk}\label{Tits-points}
	This remark will not be used anywhere in this part, so the reader may want to skip it for now and return when needed. Fix $w\in W$. It is not true that $\sigma(w^{-1})=\sigma(w)^{-1}$ in general, but the discrepancy can be computed easily. Indeed, let $\tau_w= (\sigma(w^{-1})\sigma(w))^{-1}$. Then
	\[\tau_w =\prod_{\alpha\in \Phi_{w}}\alpha^\vee(-1)\]
	We prove this statement by induction. Choose a pinning $\{\iota_\alpha\}_{\alpha\in \Delta^+}$ of $G$ such that $\sigma(\alpha) = \iota_\alpha\left(\begin{smallmatrix}0 & 1 \\ -1 & 0\end{smallmatrix}\right)$. For a simple root $\alpha$ we have
	\[\tau_{s_\alpha} = (\sigma(\alpha)\sigma(\alpha))^{-1} = \iota_\alpha \left(\begin{matrix}
		0 & 1 \\ -1 & 0
	\end{matrix}\right)^{-2} = \alpha^\vee(-1) \]
	where $\iota_\alpha:\SL_2\to G_\alpha$ is given by the pinning. This is the desired property, as $\Phi_{s_\alpha^{-1}}= \Phi_{s_\alpha} = \{\alpha\}$. In addition, note that $(\tau_\alpha)^{-1}= \tau_\alpha$.

	Now let $w=w' s_{\alpha}$ with $\alpha$ simple and $\ell(w)=\ell(w')+1$. Then $\sigma(w)=\sigma(w')\sigma(s_\alpha)$, $\sigma(w^{-1})=\sigma(s_\alpha)\sigma(w'^{-1})$ and
	\begin{align*}
		\tau_w^{-1}&=\sigma(w^{-1})\sigma(w) = \sigma(s_\alpha)\tau_{w'}^{-1}\sigma(s_\alpha) = ({}^{s_\alpha}\tau_{w'}^{-1})\, \tau_\alpha^{-1}
	\end{align*}
	Using the inductive hypothesis, we get
	\begin{align*}
		\tau_w = \left(\prod_{\beta\in \Phi_{w'}}(s_\alpha\cdot\beta)^\vee(-1)\right) \alpha^\vee(-1)
	\end{align*}
	which implies the desired property as $\Phi_{w} = s_\alpha(\Phi_{w'}) \sqcup \{\alpha\}$.

	Another useful way to compute $\tau_w$ is as follows. Let us consider
	\[ t = \prod_{\alpha\in \Phi^+} \alpha^\vee(i).\]
	We have
	\begin{align*}
		{}^wt &=  \prod_{\alpha\in \Phi^+} (w\cdot \alpha)^\vee(i) = \left(\prod_{\alpha \in \Phi^+\setminus \Phi_{w^{-1}}} \alpha^\vee(i)\right)\left(\prod_{\alpha \in \Phi_{w^{-1}}} \alpha^\vee(i)^{-1}\right)
	\end{align*}
	and, hence,
	\[{}^w \tau_w = t({}^wt)^{-1}.\]
\end{rmk}

\subsection{Walk stratification}\label{sec:walk-stratification}

Let $\beta\in \Br$ be a positive braid and $w\in W$. In this subsection we will stratify
\[\rho(\beta)_{w} := \{z\in \rho(\beta): f(z)\in BwB\}\]
where $f:\rho(\beta)\to G$ is the structure morphism. This will depend, in principle, on the choice of an expression $\beta = s_{\alpha_l}\cdots s_{\alpha_1}$ as a product of simple reflections. This follows a similar approach to \cite{Mellit} for $\GL_n$.

A walk $p$ is a sequence $p_0,\dots,p_{l}$ of elements of $W$ such that
\[p_{k}=\left\{\begin{array}{cc}
        s_{\alpha_{k}}p_{k-1} &  \text{if }s_{\alpha_{k}}p_{k-1}>p_{k-1}\\
        s_{\alpha_{k}}p_{k-1}\text{ or }p_{k-1} & \text{if not.}
    \end{array}\right.\]
with respect to the Bruhat order. In each case, we say that we go up ($p_{k}>p_{k-1}$), go down ($p_{k}<p_{k-1}$), or stay ($p_{k}=p_{k-1}$) at $k$. Let $U_p$, $D_p$, and $S_p$ be the positions where we go up, go down, or stay, respectively. Let $\mcW_{w}(\beta)$ be the set of walks starting at the identity $e$, $p_0=e$, and finishing at $w$, $p_l=w$.  To each such walk, there is an associated $B$-stable subset of $\rho(\beta)_w$
\[ C_p := \{ (z_l,\ldots,z_1)\in \rho(\beta): z_k\cdots z_1\in Bp_kB \text{ for each } k = 1,\ldots,l\}.\]
Each $C_p$ is a twisted $B$-variety.

\begin{thm}\label{thm:walk-str}
    Let $\beta\in \Br$ be a positive braid and $w\in W$. Choose an expression $\beta = s_{\alpha_l}\cdots s_{\alpha_1}$ in terms of simple reflections. Then there is a weak $U$-decomposition of $\rho(\beta)_{w}$ indexed by $\mcW_{w}(\beta)$ where the cell $C_p$ associated with a walk $p$ is $U$-isomorphic as a twisted $T$-variety to
    \[\rho(w) * \left(\prod_{k\in D_p} \tau_{p_k^{-1}\cdot \alpha_k} \underline{U_{p_k^{-1}\cdot \alpha_k}}\right) * \left(\prod_{k\in S_p} T_{p_k^{-1}\cdot \alpha_k}\right)\]
    where all convolutions are made from right to left in increasing order for $k$.
\end{thm}
\begin{proof}
    It follows by computing the convolutions from right to left, iteratively applying Lemma \ref{lem:going-up-B}, when we go up, or the exchange condition and the following generalization of Lemma \ref{lem:not-going-up2}, when we either stay or go down.
\end{proof}

\begin{lema}
    Let $\alpha$ be a simple root, and $w\in W$ such that $w^{-1}\cdot \alpha < 0$. Let $w'=s_\alpha w\in W$. Then, for any twisted $B$-variety $X$, $X * \rho(s_\alpha) * \rho(w)$ weakly $U$-decomposes as a disjoint union of $X * \rho(w') *  \tau_{w'^{-1}\cdot \alpha}\underline{U_{{w'^{-1}\cdot \alpha}}}$ and $X * \rho(w)* T_{w^{-1}\cdot \alpha} $.
\end{lema}
\begin{proof}
    By Theorem \ref{thm:braid-var}, we know that $\rho(w) = \rho(s_\alpha) * \rho(w')$ for $w'=s_\alpha w\in W$. By Lemmas \ref{lem:not-going-up2} and \ref{lem:U-conv}, we get a $U$-decomposition into $X * \tau_\alpha * U_{s_\alpha}^- * \rho(w')$ and $X * T_\alpha * \rho(w)$. Both torus parts can be moved to the right using Lemma \ref{lem:commute}. What is left to do is to prove that $ U_{s_\alpha}^- * \rho(w')$ is $U$-isomorphic as a $T$-variety to $\rho(w') *  \underline{U_{{w'^{-1}\cdot \alpha}}}$. A $U$-isomorphism is $(x,y)\mapsto (y, \sigma(w'^{-1})x\sigma(w'^{-1})^{-1})$.
\end{proof}

\subsection{Stabilizers}\label{subsec:stabs}

In this subsection we describe the stabilizers of the $T$-action on a cell $C_p$ associated with $\beta\in \Br$ and a walk $p$, and provide a finiteness criterion. Let $Q=\langle \Delta^+ \rangle_\Z\subset \mathscr{X}$ be the lattice of roots. We define $Q(p)\subset Q$ as the $\Z$-module generated by $\pi_k^{-1}\cdot\alpha_k$, $k\in S_p$, where $\pi_k=s_{\alpha_k}\cdots s_{\alpha_1}\in W$. We define the center of $p$ as $Z(p):=(Q/Q(p))^*=\Hom{Q/Q(p)}{\C^\times}$. Let $S(p)\subset T$ be the subgroup generated by the stabilizers in $T$ of all points of $C_p$. We denote by $Z(G)$ the center of $G$.

\begin{thm}\label{thm:stabilizers}
    Let $\beta\in \Br$ and $p$ be an associated walk. Then $Z(G)\subset S(p)$ and $S(p)/Z(G)\simeq Z(p)$.
\end{thm}
\begin{proof}
    Note that $S(p)$ agrees with the stabilizer of any point of the torus part of $C_p$. To understand how $t\in T$ acts, we need to perform the convolution. By its very definition, $t\in T$ acts on the $k$-th coordinate by multiplication by
    \[ (p_k^{-1}\cdot \alpha_k)\left({}^{s_{p_{k_n}^{-1}\alpha_{k_n}}\cdots s_{p_{k_1}^{-1}\alpha_{k_1}}}t\right)^{-1} = \left(\left(s_{p_{k_1}^{-1}\alpha_{k_1}}\cdots s_{p_{k_n}^{-1}\alpha_{k_n}}p_k^{-1} \right)\cdot \alpha_k\right)(t)^{-1}\]
    where $k_1<k_2<\ldots<k_n<k$ are the previous stays. We claim this agrees with $(\pi_k^{-1}\cdot \alpha_k)(t)$. It suffices to show that
    \[ s_{p_{k_1}^{-1}\alpha_{k_1}}\cdots s_{p_{k_n}^{-1}\alpha_{k_n}} = \pi_k^{-1}s_{\alpha_k}p_k .\]
    We proceed by induction. This is obviously true for the first stay as $\pi_k = s_{\alpha_k}p_k$ in this case. Hence, it suffices to show that, if $k$ is a stay, not the first one, and $k'$ is the previous stay, then
    \[ \pi_{k'}^{-1} s_{\alpha_{k'}} p_{k'} s_{p_{k'}^{-1}\alpha_{k'}} = \pi_k^{-1}s_{\alpha_k}p_k .\]
    Note that the left-hand side is $\pi_{k'}^{-1} p_{k'}$ as $s_{p_{k'}^{-1}\alpha_{k'}} = p_{k'}^{-1}s_{\alpha_{k'}}p_{k'}$. Now the equality follows from the very definition of $p_k$ and $p_{k'}$.

    Then $x\in T$ belongs to $S(p)$ if and only if $({\pi_k^{-1}}\cdot \alpha_k)(x)=1$ for all $k\in S_p$. Let $\mcE:\t \to T$ be the exponential $x\mapsto \exp(2\pi i x)$. We recall that
    \[\mcE^{-1}(Z(G))=P^\vee:=\{t\in \t:\alpha(t)\in \Z\ \forall\ \alpha\in Q\}\]
    see for example \cite[Theorem 7, 4.3.5]{OV}. As $G$ is connected, any $x$ is of the form $\mcE(y)$ for some $y\in \t$. Now, $x\in S(p)$ if and only if $\mcE((\pi_k^{-1}\alpha_k)(y))=({\pi_k^{-1}}\alpha_k)(x)=1$ for all $k\in S_p$, in other words if $(\pi_k^{-1}\alpha_k)(y)\in \Z$ for all $k\in S_p$. Equivalently $y\in P(p)^\vee:=\{t\in \t:\alpha(t)\in \Z\ \forall\ \alpha\in Q(p)\}$. Therefore
    \[ S(p)/Z(G)\simeq \mcE(P(p)^\vee)/\mcE(P^\vee)\simeq P(p)^\vee/P^\vee\simeq (Q/Q(p))^*\]
    where the last map is given by sending $t\in P(p)^\vee$ to the evaluation map $Q\to \C$ on $t$ followed by exponentiation.
\end{proof}

For the finiteness criterion we need a few definitions. Let $w\in W$ be the end of the walk $p$. This means that the structure morphism $C_p\to G$ has image in $BwB$. Using the identification $U\times T\times U_w^-\simeq BwB$, $(u_1,t,u_2)\mapsto u_1 \sigma(w)t u_2$, we get a map $g: C_p \to T$. Let $G^{ss} = [G, G]$ be the semisimple part of $G$ and $T^{ss} = T\cap G^{ss}$.

\begin{thm}\label{thm:stab-finite}
    Let $\beta\in \Br$ and $p$ be an associated walk. Then $g:C_p\to T$ has image in $T^{ss}$ and $Z(p)$ is finite if and only if $g: C_p\to T^{ss}$ is surjective.
\end{thm}

To show this result, we first note, from the proof of Theorem \ref{thm:stabilizers}, that the tangent algebra to $S(p)$ is the connected component of $P(p)^\vee$ that contains zero and that $Z(p)$ is finite if and only if $Q\otimes\C = Q(p)\otimes \C$. So we look to relate $Q(p)\otimes \C$ with the tangent algebra to the image of $g$. To this end, let, for each root $\alpha$, $h_\alpha:= d\alpha^\vee(1)\in \t$ be its associated semisimple element. Note that ${}^wh_\alpha = h_{w\cdot \alpha}$. By this very definition, the complexification of
\[ Q^\vee(p):=\langle {}^{p_k^{-1}}h_{\alpha_k}:k\in S_p\rangle_\Z \]
is the desired tangent algebra, up to a translation. In addition, note that the tangent algebra to $T^{ss}$ is the span of $h_\alpha$ over all roots $\alpha$, that is, the lattice generated by the duals of the elements of $Q$. Thus, Theorem \ref{thm:stab-finite} follows from Lemma \ref{lem:Qvee} below.

\begin{lema}\label{lem:Qvee}
	$Q^\vee(p)$ is generated by the duals of the elements of $Q(p)$, in other words, $Q^\vee(p)=\langle {}^{\pi_k^{-1}}h_{\alpha_k}:k\in S_p\rangle_\Z$.
\end{lema}

Before showing this lemma, we look at the symmetries of these spaces.

\begin{lema}\label{lem:symmetries}
	For any walk $p$,
 	\[W(p):=\langle s_{\pi_k^{-1}\alpha_k}:k\in S_p\rangle =\langle s_{p_k^{-1}\alpha_k}:k\in S_p\rangle\]
	as subgroups of $W$. In addition, $\pi_k^{-1}p_k\in W(p)$ for any $k$.
\end{lema}
\begin{proof}
	We show this by induction on $k$. For the first stay we have $s_{\pi_k^{-1}\alpha_k} = \pi_k^{-1}s_{\alpha_k}\pi_k = p_k^{-1}s_{\alpha_k}^3p_k = p_k^{-1}s_{\alpha_k}p_k = s_{p_k^{-1}\alpha_k}$. Now assume that the result holds for the stays $k_1<\ldots<k_n$ previous to $k$, that is,
	\[\langle s_{\pi_{k_i}^{-1}\alpha_{k_i}}:1\leq i\leq n\rangle =\langle s_{p_{k_i}^{-1}\alpha_{k_i}}:1\leq i\leq n\rangle.\]
	Recall from the proof of Theorem \ref{thm:stabilizers} that $\theta_k=\pi_k^{-1}s_{\alpha_k}p_k$ belongs to this group. Hence, the inductive step follows by noting that $\theta_k^{-1}s_{\pi_k^{-1}\alpha_k}\theta_k = s_{p_k^{-1}\alpha_k}$.
	The last claim follows by noticing that, for any $k$, $\pi_k^{-1}p_k = \theta_{k'}$ for the first stay $k'>k$.
\end{proof}

\begin{proof}[Proof of Lemma \ref{lem:Qvee}]
    We note that $\langle {}^{\pi_k^{-1}}h_{\alpha_k}:k\in S_p\rangle$ and $Q^\vee(p)=\langle {}^{p_k^{-1}}h_{\alpha_k}:k\in S_p\rangle$ are invariant under the action of $W(p)$. Now, as $\theta_k=\pi_k^{-1}s_{\alpha_k}p_k\in W(p)$, ${}^{\pi_k^{-1}}h_{\alpha_k}=-{}^{\theta_kp_k^{-1}}h_{\alpha_k}\in Q^\vee(p)$ and vice versa.
\end{proof}

\begin{ex}\label{ex:bad-cell}
    It may happen that $Z(p)$ is finite but non-trivial. For example, let $G=\Sp_4(\C)$ for the bilinear form
    \[\left(\begin{array}{cccc}
        0 & 0 & 1 & 0 \\
        0 & 0 & 0 & 1 \\
        -1 & 0 & 0 & 0 \\
        0 & -1 & 0 & 0
    \end{array}\right)\]
    see \cite[Example 4.1.5.8]{OV}. Then $Z(G)= \{\pm\operatorname{Id}_4\}$ and $T=\{\operatorname{diag}(x,y,x^{-1},y^{-1}):x,y\in \C^\times\}$ is a maximal torus. Let $\alpha_1=2E_{22}^*$ and $\alpha_2=E_{11}^*-E_{22}^*$. We consider the following walk:
    \begin{table}[H]
        \centering
        \begin{tabular}{c|cccccccc}
            $k$ & $1$ & $2$ & $3$ & $4$ & $5$ & $6$ & $7$ & $8$  \\
            \hline
            root & $\alpha_1$ & $\alpha_1$ & $\alpha_1$ & $\alpha_2$ & $\alpha_1$ & $\alpha_1$ & $\alpha_1$ & $\alpha_2$ \\
            step & up & stay & down & up & up & stay & down & down \\
            $p_k$ & $s_{\alpha_1}$ & $s_{\alpha_1}$ & $e$ & $s_{\alpha_2}$ & $s_{\alpha_1}s_{\alpha_2}$ & $s_{\alpha_1}s_{\alpha_2}$ & $s_{\alpha_2}$ & $e$ \\
            $\pi_k$ & $s_{\alpha_1}$ & $e$ & $s_{\alpha_1}$ & $s_{\alpha_2}s_{\alpha_1}$ & $s_{\alpha_1}s_{\alpha_2}s_{\alpha_1}$ & $s_{\alpha_2}s_{\alpha_1}$ & $s_{\alpha_1}s_{\alpha_2}s_{\alpha_1}$ & $s_{\alpha_2}s_{\alpha_1}s_{\alpha_2}s_{\alpha_1}$
        \end{tabular}
        \caption{A walk with non-trivial but finite center $Z(p)$ on $\Sp_4(\C)$.}
    \end{table}

    It has
    \[Q^\vee(p)=\langle {}^{s_{\alpha_1}}h_{\alpha_1}, {}^{s_{\alpha_2}s_{\alpha_1}}h_{\alpha_1}\rangle =\langle \operatorname{diag}(0,1,0,-1), \operatorname{diag}(1,0,-1,0)\rangle\]
    and
    \[S(p)=\ker(2E_{22}^{*}) \cap \ker (2E_{11}^{*})\simeq (\Z/2\Z)^2.\]
    In this case, $Z(p)=\Z/2\Z$.

    Similar examples can be built for types $B_n$, $C_n$, $F_4$ and $G_2$ to get interesting $Z(p)$. However, in Theorem \ref{thm:no-quotients-typeA-unipvar}, we will show that this is not the case for type $A_n$.
\end{ex}

\subsection{Torus fibers}\label{sec:tori-fibers} We finish this part by describing the fibers of $g:C_p\to T$ when $Z(p)$ is finite. We define the fundamental group of $p$ as $\pi_1(p):=\t^{ss}(\Z)/Q^\vee(p)$, where $\t^{ss}(\Z)$ is the cocharacter lattice of $T^{ss}$.

\begin{thm}\label{thm:tori-fibers}
    Let $\beta\in \Br$ and $p$ be an associated walk. If $g:C_p\to T^{ss}$ is surjective, then $\pi_1(p)$ is finite and \[g^{-1}(t)\simeq \C^{|U_p|}\times (\C^\times)^{|S_p|-\dim T^{ss}}\times \pi_1(p)\]
    for any $t\in T^{ss}$.
\end{thm}
\begin{proof}
    We note that, by its very definition, $\coker(g_*:\pi_1(C_p)\to \pi_1(T^{ss})) = \t^{ss}(\Z)/Q^\vee(p)=\pi_1(p)$ and it is finite as $Q^\vee(p)\otimes\C=\t^{ss}$. Now, calling $w$ the end of $p$, $g:\sigma(w)U_w^-\times \C^{|D_p|}\times (\C^\times)^{|S_p|}\to T$ is the composition of the projection onto the torus component, a translation, and a group homomorphism $g'$. It follows that $g^{-1}(t)\simeq \C^{|U_p|}\times \ker g'$ and that the identity component of $\ker g'$ is a $|S_p|-\rk Q^\vee(p)$ dimensional torus. Finally, by \cite[Theorem 4, 1.3.4]{OV}, $\pi_0(\ker g')\simeq \coker(g_*)=\pi_1(p)$.
\end{proof}

\part{Stratification of character stacks}

\section{Strategy}\label{sec:strategy}

In this section, we will explain how to stratify character stacks. Let us start by defining them. Fix a connected and reductive algebraic group $G$ from now on. Given a non-negative integer $g$ and a tuple $\mcC=(\mcC_1,\ldots, \mcC_m)$ of conjugacy classes of $G$, the associated character stack $\abbreviatedBetti[G]$ is defined as the quotient stack
\[ [\{(x_1,y_1,\ldots,x_g,y_g,z_1,\ldots,z_m)\in G^{2g}\times \mcC_1\times \dots\times \mcC_m: [x_1,y_1]\cdots[x_g,y_g]z_1\cdots z_m= 1 \}/ (G/Z(G))]\]
where $[x,y]=xyx^{-1}y^{-1}$ and $G$ acts by simultaneous conjugation on every variable.

The stacks $\abbreviatedBetti[G]$ can be written in terms of simpler pieces using a field theory\footnote{We are not going to actually define a field theory because our results do not formally need it. Nonetheless, we mention that it can be done either by adding basepoints to the surfaces, see for example \cite{GP}, or by restricting the class of bordisms allowed.} $\mcF$ in the sense of \cite{GPHV} and \cite{Jesse}. The idea is that any connected closed orientable surface $M$ with punctures can be built by gluing the simpler pieces of Figure \ref{fig:bordisms}. Thus, one expects a similar decomposition for the character stacks. To formalize this, we need to introduce the category of correspondences $\Corr[G]$.

\begin{figure}[H]
	\centering
	\includegraphics[]{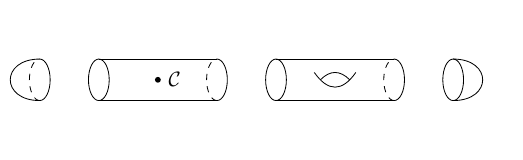}
	\caption{From left to right: the cap $D$, a punctured cylinder $L_\mcC$, the handle $H$, and the cup $D^\dagger$.}
	\label{fig:bordisms}
\end{figure}

The objects of $\Corr[G]$ are $G$-varieties. A morphism between $X$ and $Y$ is a correspondence
 \[
 \xymatrix{
	& Z \ar^{h}[rd]\ar_{f}[ld]& \\
	X & & Y
	}
\]
where $Z$ is another $G$-variety and $f$ and $h$ are $G$-equivariant. Composition is given by pullback; the composition of $(Z,f,h):X\to Y$ and $(Z',f',h'):Y\to Y'$ is
\[\xymatrix{
	& & Z\times_Y Z'\ar[ld]\ar[rd] & & \\
	& Z \ar_{f}[ld]\ar^{h}[rd] & & Z' \ar_{f'}[ld]\ar^{h'}[rd]&\\
	X & & Y & & Y'
}\]
viewed as a correspondence between $X$ and $Y'$.

To each surface of Figure \ref{fig:bordisms} we associate a particular correspondence. The easiest ones are the cap and the cup. Let $\iota:\{\mathrm{pt}\}\to G$ be the inclusion of the unit $e$ of $G$. Define $\mcF(D)$ and $\mcF(D^\dagger)$ as the correspondences $(\{\mathrm{pt}\},\operatorname{id},\iota)$ and $(\{\mathrm{pt}\},\iota,\operatorname{id})$. To the handle we associate the span $\mcF(H)$
\[\xymatrix{
		& G\times G\times G\ar[ld]\ar[rd] & \\
		G & & G
	}\]
where the left arrow is the projection onto the last coordinate and the right arrow is $(x,y,a)\mapsto [x,y]a$. Finally, for a conjugacy class $\mcC\subset G$, $\mcF(L_\mcC)$ is
\[\xymatrix{
	& \mcC\times G\ar[ld]\ar[rd] & \\
	G & & G
}\]
where the left arrow is again the projection onto the last coordinate and the right arrow is the multiplication. All the $G$-actions are given by conjugation. In this way,
\[\abbreviatedBetti[G] = [(\mcF(D^\dagger)\circ \mcF(H)^g\circ \mcF(L_{\mcC_1})\circ\dots\circ \mcF(L_{\mcC_m})\circ \mcF(D) ) / (G/Z(G))].\]

These correspondences are hard to compute explicitly in general, see \cite{GP} and the references therein for some computations. To overcome this, we assume $\mcC_m$ is the conjugacy class of a regular semisimple element whose centralizer is connected. In this case, there is a simpler description of the character stack. Fix a maximal torus $T$ of $G$ and a representative $C_m$ of $\mcC_m$ in $T$. As it is regular semisimple and its centralizer is connected, $Z(C_m)$ agrees with $T$. Then
\[\abbreviatedBetti[G]\simeq [\{(x,y,z)\in G^{2g}\times \mcC_1\times \dots\times \mcC_{m-1}: [x_1,y_1]\cdots[x_g,y_g]z_1\cdots z_{m-1}C_m= 1 \}/ (T/Z(G))].\]

Furthermore, fix a Borel subgroup $B$ of $G$ containing $T$ and let $U$ be its unipotent radical. The conjugation action of $T$ preserves $U$, hence we can define
\[\tBetti[G]:= [\{(x,y,z)\in G^{2g}\times \mcC_1\times \dots\times \mcC_{m-1}: [x_1,y_1]\cdots[x_g,y_g]z_1\cdots z_{m-1}C_m\in U \}/ (T/Z(G))]\]
where $T$ acts by simultaneous conjugation on every variable. Equivalently,
\[\tBetti[G]\simeq [\{(x,y,z,u)\in G^{2g}\times \mcC_1\times \dots\times \mcC_{m-1}\times U: [x_1,y_1]\cdots[x_g,y_g]z_1\cdots z_{m-1}uC_m = 1 \}/ (T/Z(G))]\]
where $T$ acts on $U$ by conjugation. This stack turns out to be (almost) a fibration over the character stack, see Section \ref{sec:bundle-structure}. The point is that
\[\abbreviatedBetti[G] \simeq [\{(x,y,z,u)\in G^{2g}\times \mcC_1\times \dots\times \mcC_{m-1}\times U: [x_1,y_1]\cdots[x_g,y_g]z_1\cdots z_{m-1}uC_m = 1\}/ (B/Z(G))]\]
for the $B$-action $b\cdot (x,y,z,u)= (bxb^{-1},byb^{-1},bzb^{-1},buC_mb^{-1}C_m^{-1})$.

We can work with $\tBetti[G]$ using the field theory $\mcF$. Consider the quantization functor $\mcQ: \Corr[G]\to \mathrm{Cat}$ to the category of categories, which sends a $G$-variety $X$ to the category $B\Var_X$ of $B$-varieties over it, and a correspondence $(Z,f,h):X\to Y$ to $h_!f^*:B\Var_X \to B\Var_Y$, where $f^*:B\Var_X\to B\Var_Z$ is given by base change and $h_!:B\Var_Z\to B\Var_Y$ by composition with $h$. Let $\mcZ$ be the composition of $\mcF$ and $\mcQ$. By definition, $\mcZ(H)$ and $\mcZ(L_\mcC)$ have domain $B\Var_G$. Now, we can view $UC_m$ as a $B$-variety over $G$ via the inclusion $UC_m\subset G$ and the conjugation action of $B$. It turns out that $\tBetti[G]$ is
\[  [(\mcZ(D^\dagger)\circ\mcZ(H)^{g}\circ\mcZ(L_{\mcC_{1}})\circ\dots\circ \mcZ(L_{\mcC_{m-1}}))(UC_m)/ (T/Z(G))]\]
while $\abbreviatedBetti[G]$ is the quotient by $B/Z(G)$.

There is one more operation we want to consider: the convolution $*:B\Var_G\times B\Var_G \to B\Var_G$ defined by $(X,f)* (Y,f')=(X\times Y, \mu(f,f'))$, where $\mu$ is the multiplication of $G$. It is the value of $\mcZ$ at the reverse pair of pants (see Figure \ref{fig:reverse-pairs-of-pants}). We can think about the punctured cylinders and the handle operation as convolutions with fixed varieties: $\mcF(L_\mcC)\circ\mcF(D)$ and $\mcF(H)\circ \mcF(D)$ (Figure \ref{fig:reiterpretations}). More concretely, these are the inclusion of $\mcC$ in $G$ and the commutator map $G\times G\to G$, $(x,y)\mapsto [x,y]$, respectively.

\begin{figure}[H]
    \centering
    \includegraphics[]{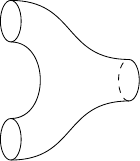}

    \caption{The reverse pair of pants.}
    \label{fig:reverse-pairs-of-pants}
\end{figure}

\begin{figure}[H]
    \centering
    \includegraphics[]{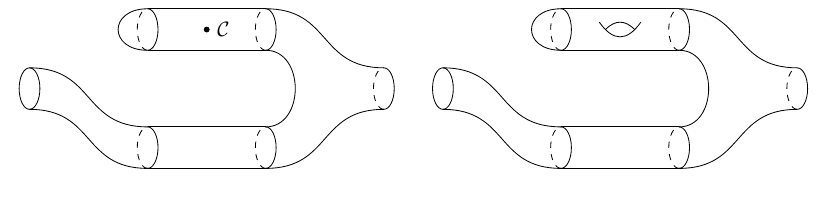}

    \caption{From left to right: reinterpretations of a punctured cylinder $L_\mcC$ and of the handle $H$.}
    \label{fig:reiterpretations}
\end{figure}

Now, both $\mcC$ and the commutator map can be decomposed by the Bruhat decomposition of $G$. What we are going to show in Sections \ref{sec:handle-operation} and \ref{sec:adding-punctures} is that the convolution of these cells with $U$ can be described as convolutions of $U$ with the braid varieties and a few simpler pieces that we introduced in Part \ref{part-one}. A remarkable fact is that this process can be easily iterated because convolution is associative and the resulting varieties always have a convolution with $U$ at the very left. By this procedure we will get a cell decomposition of $(\mcZ(H)^{g}\circ\mcZ(L_{\mcC_{1}})\circ\dots\circ \mcZ(L_{\mcC_{m-1}}))(UC_m)$. This decomposition can then be further refined using the walk stratification of Theorem \ref{thm:walk-str}.

The next step is to apply $\mcZ(D^\dagger)$ and quotient by $T$. The former is performed in a similar fashion to Subsection \ref{sec:tori-fibers}. To deal with the quotient by $T$, we will explore the stabilizers of its action on each cell in Section \ref{sec:quotient-by-T}. We are going to show that the stabilizers in $T/Z(G)$ are finite if and only if the image of the cell in $T$ is as big as possible. Hence, for generic $\mcC$, the stabilizers will be finite and we will be able to perform the quotient to get a nice decomposition of $\tBetti[G]$. Finally, we will show that each cell is a preimage by $\tBetti[G]\to \abbreviatedBetti[G]$.

\section{Handle operator}\label{sec:handle-operation}

We follow the conventions and notation of the previous part. In particular, we fix once for all a choice $\sigma:\Delta^+\to G$ as in Section \ref{sec:unipotent-varieties}. Given $w_1,w_2\in W$, we define a twisted $B$-variety $T(w_1,w_2)$ as follows. As a $T$-variety, it is $T(w_1)\times T(w_2)$, but its structure morphism is
\[ (t_1,t_2)\mapsto {}^{w_2w_1w_2^{-1}w_1^{-1}}t_1^{-1}{}^{w_2w_1w_2^{-1}}t_2^{-1}{}^{w_2}t_1t_2 \]
for $t_1,t_2\in T$, and the $B$-action is the induced one. It is simple to check that this is a $[w_2,w_1]$-twisted $B$-variety.

\begin{thm}\label{thm:handle}
	For any $e$-twisted $B$-variety $X$, $\mcZ(H)^{g}(\mcU(X))$ has a weak decomposition indexed by $W^{2g}$ such that the cell associated to $(\omega_1^{(i)},\omega_2^{(i)})_{i=1}^g$ is
	\[  \mcU\left(\rho\left(\left[\omega_1^{(g)},\omega_2^{(g)}\right]\cdots \left[\omega_1^{(1)},\omega_2^{(1)}\right]\right)* X* \prod_{i=1}^g T\left(\omega_1^{(i)},\omega_2^{(i)}\right)\right)* \left(\prod_{i=1}^g\underline{U_{\omega_1^{(i)}}^+} * \underline{U_{\left(\omega_2^{(i)}\right)^{-1}}^+}\right)\]
	where the commutators are taken as words in $\Br$ and the convolution of the tori is made from left to right in increasing order of $i$.
\end{thm}
\begin{proof}
	This follows from the Bruhat decomposition and Lemma \ref{lem:handle-computation} below.
\end{proof}

\begin{rmk}\label{rmk:handle-tori}
	If one is willing to forget the $T$-action, this decomposition can be simplified further. Indeed, the unipotent pieces are just affine spaces and $T(w_1,w_2)$ turns out to be $T^2\to T$, $(t_1,t_2)\mapsto t_1{}^{w_1}t_1^{-1}t_2{}^{w_2}t_2^{-1}$. Explicitly, an isomorphism over $G$ can be built as follows. First, change $t_2$ by ${}^{w_1^{-1}}t_1t_2$. This changes the map to $G$ to
	\[ {}^{w_2}t_1{}^{w_1^{-1}}t_1^{-1}t_2{}^{w_2w_1w_2^{-1}}t_2^{-1} \]
	Now, conjugate $t_1$ by $w_2$ to get
	\[ t_1{}^{w_1^{-1}w_2^{-1}}t_1^{-1}t_2{}^{w_2w_1w_2^{-1}}t_2^{-1} \]
	Continue by multiplying $t_1$ by ${}^{w_2w_1w_2^{-1}}t_2^{-1}$ to get
	\[t_1{}^{w_1^{-1}w_2^{-1}}t_1^{-1}t_2{}^{w_2^{-1}}t_2^{-1}\]
	Now change $t_1$ by ${}^{w_1^{-1}w_2^{-1}}t_1^{-1}$ and $t_2$ by ${}^{w_2^{-1}}t_2^{-1}$ so that the function to $G$ becomes
	\[t_1{}^{w_2w_1}t_1^{-1}t_2{}^{w_2}t_2^{-1}\]
	Finally, multiply $t_2$ by ${}^{w_1}t_1$.
\end{rmk}

\begin{lema}\label{lem:handle-computation}
    Let $w_1,w_2\in W$ and $f:X\to G$ be an $e$-twisted $B$-variety. Consider the twisted $B$-variety over $G$ given by $Bw_1B \times Bw_2B \times \mcU(X) \to G$, $(a,b,u,x)\mapsto [a,b]uf(x)$, and the $B$-action given by convolution. Then it is isomorphic to
    \[  \mcU(\rho(w_1)* \rho(w_2)* \rho(w_1^{-1})* \rho(w_2^{-1})* X* T(w_1, w_2))* \underline{U_{w_1}^+} * \underline{U_{w_2^{-1}}^+}\]
    as a twisted $T$-variety.
\end{lema}

\begin{rmk}\label{rmk:handle-B}
	As the proof below shows, the isomorphism fails to be $B$-equivariant only on the factors $\underline{U_\omega^+}$.
\end{rmk}

\begin{proof}
    We construct a $B$-equivariant map
    \[Bw_1B \times Bw_2B \times \mcU(X) \to \mcU(T(w_1^{-1})* \rho(w_1)*T(w_2^{-1})* \rho(w_2)*T(w_1)* \rho(w_1^{-1})*T(w_2)* \rho(w_2^{-1})*X)\]
    over $G$ as follows. First, consider the embedding
    \[ Bw_1B \times Bw_2B \times \mcU(X) \to Bw_1B * Bw_2B * Bw_1^{-1}B * Bw_2^{-1}B * \mcU(X)\]
    given by $(a,b,u,x)\mapsto (a,b,a^{-1},b^{-1},u,x)$. Now, recall that $BwB \simeq \mcU(T(w^{-1})* \rho(w))$ as twisted $B$-varieties for any $w\in W$. Thus, monoidality of the untwisting functor yields the desired morphism, see Lemma \ref{lem:magic-B-equivariancy}(3). Explicitly, denote by $\pi_w^+: BwB \to U$ and $\pi_w^-:BwB\to T(w^{-1})*\rho(w)$ the unique maps such that $y=\pi_w^+(y)\pi_w^-(y)$ for $y\in BwB$. Let
    \begin{align*}
        \psi_0(a,b,u) &= b^{-1}u, & && \Psi_0(a,b,u)&= \pi_{w_2^{-1}}^-(\psi_0(a,b,u)), \\
        \psi_1(a,b,u) &= a^{-1}\pi_{w_2^{-1}}^+(\psi_0(a,b,u)) , & && \Psi_1(a,b,u)&= \pi_{w_1^{-1}}^-(\psi_1(a,b,u)), \\
        \psi_2(a,b,u) &= b\pi_{w_1^{-1}}^+(\psi_1(a,b,u)), & && \Psi_2(a,b,u)&= \pi_{w_2}^-(\psi_2(a,b,u)), \\
        \psi_3(a,b,u) &= a\pi_{w_2}^+(\psi_2(a,b,u)), & && \Psi_3(a,b,u)&= \pi_{w_1}^-(\psi_3(a,b,u)), \\
        \text{and} && && \Psi_4(a,b,u) &= \pi_{w_1}^+(\psi_3(a,b,u)).
    \end{align*}
    Then the morphism is $(\Psi_4, \Psi_3, \Psi_2, \Psi_1, \Psi_0) \times \operatorname{id}_X$.

    Let us write $a = u_1  t_1 \sigma(w_1) v_1$ and $b =  v_2 t_2 \sigma(w_2) u_2$ where $u_1 \in U_{w_1^{-1}}^-$, $u_2 \in U_{w_2}^-$, $v_i \in U$, and $t_i\in T$. Furthermore, write $v_1v_2 = z_1^+ z_1^-$, with $z_1^+\in U_{w_2^{-1}}^+$ and $z_1^-\in U_{w_2^{-1}}^-$, and $u_2v_1^{-1} = z_2^+z_2^-$, with $z_2^+\in U_{w_1}^+$ and $z_2^-\in U_{w_1}^-$.  Note that these decompositions are unique. We claim that with all $\Psi_i$ and $z_1^+$ and $z_2^+$ we can recover $a$, $b$ and $u$. For this, we look at the partial products:
    \begin{align*}
        \Psi_4 \Psi_3 &= a \pi_{w_2}^+(\psi_2)\\
        \Psi_4 \Psi_3\Psi_2 &= ab \pi_{w_1^{-1}}^+(\psi_1)\\
        \Psi_4 \Psi_3\Psi_2\Psi_1 &= aba^{-1} \pi_{w_2^{-1}}^+(\psi_0)\\
        \Psi_4 \Psi_3\Psi_2\Psi_1\Psi_0 &= aba^{-1}b^{-1}u
    \end{align*}
    From each one we can recover partial information about $a$ and $b$. From the first one, using the same decomposition as with $a$, we get $u_1$ and $t_1$, since $\pi_{w_2}^+(\psi_2)\in U$. If we multiply the second one on the left by $(u_1t_1\sigma(w_1)z_1^+)^{-1}$ we get $ z_1^-t_2\sigma(w_2)u_2 \pi_{w_1^{-1}}^+(\psi_1)$. So, using the same decomposition as with $a$, we obtain $z_1^-$ and $t_2$. Thus, we can recover $v_1v_2$. In a similar fashion, from
    \[ (z_2^+)^{-1}(t_2\sigma(w_2))^{-1}(v_1v_2)^{-1}(u_1t_1\sigma(w_1))^{-1}\Psi_4\Psi_3\Psi_2\Psi_1 = z_2^-\sigma(w_1)^{-1}t_1^{-1}u_1^{-1}\pi_{w_2^{-1}}^+(\psi_0)\]
    we get $z_2^-$ and, thus, $u_2v_1^{-1}$. Finally, from the fourth partial product, we get
    \[ (u_1t_1\sigma(w_1))(u_2v_1^{-1})^{-1}(t_2\sigma(w_2))^{-1}(v_1v_2)^{-1}(u_1t_1\sigma(w_1))^{-1}\Psi_4\Psi_3\Psi_2\Psi_1\Psi_0 = b^{-1}u\]
    and we can recover $u_2$. We can then compute $v_2$ and $v_1$, fully recovering $a$ and $b$. To get $u$, we use the last partial product.

    Define $\Psi_5(a,b,u)=z_1^+\in U_{w_2^{-1}}^+$ and $\Psi_6(a,b,u)=z_2^+\in U_{w_1}^+$. By the previous paragraph, $\Psi = (\Psi_6, \Psi_5, \Psi_4, \Psi_3, \Psi_2, \Psi_1, \Psi_0) \times \operatorname{id}_X$ is injective. We let $T$ act on $\underline{U_{w_1}^+}$ and $\underline{U_{w_2^{-1}}^+}$ by conjugation. Then $\Psi$ is $T$-equivariant. Indeed, $t\in T$ maps $(a,b,u)$ to $(tat^{-1},tbt^{-1},tut^{-1})$, so it maps $u_i$ and $v_i$ to $tu_it^{-1}$ and $tv_it^{-1}$, by uniqueness of the writings, and hence $z_1^+$ and $z_2^+$ to $tz_1^+t^{-1}$ and $tz_2^+t^{-1}$.

    Using Lemma \ref{lem:commute} and the isomorphisms ${}^{v}T(w)\simeq T(vwv^{-1})$, $z\mapsto {}^{v}z$, we can move all tori inside the untwist to the right of $X$ to get a map, which we also call $\Psi$, to
    \begin{align*}
        &\mcU( \rho(w_1)* \rho(w_2)* \rho(w_1^{-1})* \rho(w_2^{-1})*X*T(w_2w_1w_2^{-1}w_1^{-1}w_2w_1^{-1}w_2^{-1})*T(w_2w_1w_2^{-1}w_1^{-1}w_2^{-1})\\ &*T(w_2w_1w_2^{-1})* T(w_2))* \underline{U_{w_1}^+} * \underline{U_{w_2^{-1}}^+}
    \end{align*}
    We can further simplify the torus components and still get an injective morphism. On these components, $\Psi$ is $({}^{w_2w_1w_2^{-1}w_1^{-1}}t_1,{}^{w_2w_1w_2^{-1}}t_2,{}^{w_2}({}^{w_1}\tau_{w_1}t_1^{-1}),{}^{w_2}\tau_{w_2}t_2^{-1})$, where $\tau_{w_1}$ and $\tau_{w_2}$ are as in Remark \ref{Tits-points}. Note that, by that remark and Remark \ref{rmk:handle-tori}, there is a point $(c_1,c_2)\in T(w_1,w_2)$ lying over ${}^{w_2w_1}\tau_{w_1}\,{}^{w_2}\tau_{w_2}$. Thus, we consider instead $(t_1,t_2) \mapsto (c_1t_1^{-1}, c_2t_2^{-1})\in T(w_1,w_2)$. In this way, we have built a $T$-equivariant injective morphism
    \[ Bw_1B \times Bw_2B \times \mcU(X) \to \mcU(\rho(w_1)* \rho(w_2)* \rho(w_1^{-1})* \rho(w_2^{-1})* X*T(w_1,w_2))* \underline{U_{w_1}^+} * \underline{U_{w_2^{-1}}^+}\]
    over $G$. As both varieties are abstractly isomorphic to $U_{w_1}^-\times U_{w_2}^-\times B^2\times U\times X$, it must be an isomorphism.
\end{proof}

\section{Adding good punctures}\label{sec:adding-punctures}

In this section, we will add punctures of a special type. Let $\mcC\subset G$ be the conjugacy class of a semisimple element whose centralizer is a Levi subgroup of $G$. We can always choose $C\in \mcC\cap T$ such that its centralizer $L$ is the standard Levi subgroup associated with a subset $I\subset\Delta^+$. Let $P=LB$ be the corresponding standard parabolic subgroup and $N$ its unipotent radical. We call such $C$ and $\mcC$ nice. Let $W_I\subset W$ be the subgroup generated by the symmetries associated to $I$ and identify $W/W_I$ with the set of minimal-length representatives of its cosets. Finally, recall from Remark \ref{Tits-points} that $\tau_\omega = (\sigma(\omega^{-1})\sigma(\omega))^{-1}\in T$ for $\omega\in W$.

\begin{thm}\label{thm:adding-punctures}
	Let $m\geq 2$ be an integer, $\mcC_1,\ldots,\mcC_{m-1}$ be nice conjugacy classes and $X$ be an $e$-twisted $B$-variety over $G$. For each $\mcC_i$, choose $C_i$, $I_i$ and $N_i$ as above. Then $(\mcZ(L_{\mcC_1})\circ \cdots \circ \mcZ(L_{\mcC_{m-1}}))(\mcU(X))$ has a weak decomposition indexed by $W/W_{I_1}\times\cdots\times W/W_{I_{m-1}}$ such that the cell associated to $(\omega_1,\ldots,\omega_{m-1})$ is
	\[  \mcU\left(\rho\left(\omega_1\omega_1^{-1}\cdots \omega_{m-1}\omega_{m-1}^{-1}\right)* X*\prod_{i=1}^{m-1}{}^{\omega_i}(C_i \tau_{\omega_i})\right)* \left(\prod_{i=1}^{m-1}\underline{U_{\omega_i^{-1}}^+\cap {}^{\omega_i}N_i}\right)\]
	where the first product is taken as a word in $\Br$.
\end{thm}
\begin{proof}
	This follows from the Bruhat decomposition and Lemma \ref{lem:adding-puncture-computation} below.
\end{proof}

\begin{lema}\label{lem:adding-puncture-computation}
	Let $C\in T$ be nice with centralizer $L$, $\omega\in W/W_I$, and $X$ an $e$-twisted $B$-variety over $G$. Endow $(B\omega P)/L$ with the $B$-action given by left translation and with the map to $G$ given by $x\mapsto xCx^{-1}$. Then $(B\omega P)/L *\mcU(X)$ is isomorphic to $\mcU(\rho(\omega)*\rho(\omega^{-1})*X*{}^{\omega}(C \tau_\omega))*\underline{U_{\omega^{-1}}^+\cap {}^\omega N}$ as a $T$-variety over $G$.
\end{lema}

\begin{rmk}\label{rmk:puncture-B}
	As the proof below shows, the isomorphism fails to be $B$-equivariant only on the factor $\underline{U_{\omega^{-1}}^+\cap {}^\omega N}$.
\end{rmk}

\begin{proof}
	This is essentially just the beginning of \cite[Subsection 7.3]{Mellit}. For completeness, we explain a way to build the map $T$-equivariantly and $B$-equivariantly on the untwist. First of all, note that $(B\omega P)/L \simeq U_{\omega^{-1}}^-\sigma(\omega) N$ as $T$-varieties, where $T$ acts by $t\cdot(u,n) = (t\cdot u, {}^{\omega^{-1}}t\cdot n)$. Indeed, the product map $N\times L\to P$ is bijective and $L=L$. In addition, as $\omega$ is the minimal-length representative of its coset, $U_\omega^-\subset N$. Thus, we can decompose any $n\in N$ as $n=n^+n^-$ with $n^+\in N\cap U_\omega^+$ and $n^-\in U_\omega^-$. Moreover, $n\mapsto \tilde n:=nCn^{-1}C^{-1}$ is an automorphism of the variety $N$.

	Let $c={}^\omega(C \tau_\omega)$. We claim that
	\[F:(B\omega P)/L\to \rho(\omega)*\rho(\omega^{-1})*c,\quad x = u\sigma(\omega)n\mapsto (\sigma(\omega)\tilde n^-,\sigma(\omega^{-1})cu^{-1}c^{-1},c),\]
	is a $U$-morphism. Indeed, as $\sigma(\omega^{-1})c=C\sigma(\omega)^{-1}$, the image of $F(u\sigma(\omega)n)$ in $G$ is
	\begin{align*}
		\sigma(\omega)\tilde n^-C\sigma(\omega)^{-1}u^{-1} &= \sigma(\omega)(\tilde n^+)^{-1} nCn^{-1}\sigma(\omega)^{-1}u^{-1}
	\end{align*}
	which is $xCx^{-1}$ up to a factor in $U$ on the left, as $(\tilde n^+)^{-1}\in U_\omega^+$. The map $F$ is $T$-equivariant by the description of the $T$-action above. For $U$-equivariance, given $x=u\sigma(\omega)n\in (B\omega P)/L$ and $u'\in U$, note that, by the very definition of convolution,
	\[u'\cdot F(x) = (\sigma(\omega)(\tilde n^-C \sigma(\omega)^{-1}b^{-1}\sigma(\omega)C^{-1})^-,\sigma(\omega^{-1})ca^{-1}c^{-1},c)\]
	where $u'u=a b$ with $a\in U_{\omega^{-1}}^-$ and $b\in U_{\omega^{-1}}^+$. To compute $F(u'x)$, write $\sigma(\omega)^{-1}b\sigma(\omega)=b_Lb_N$ with $b_L\in U\cap L$ and $b_N\in N$. Then $u'x=a\sigma(\omega)n'$ modulo $L$, with $n'=b_Lb_Nnb_L^{-1}\in N$. As $b_L\in L$ centralizes $C$,
	\[\tilde n' = b_Lb_N\tilde n C\sigma(\omega)^{-1}b^{-1}\sigma(\omega)C^{-1}\]
	and, as $b_Lb_N = \sigma(\omega)^{-1}b\sigma(\omega)$ and $\tilde n^+$ lie in $U_\omega^+$,
	\[\tilde n'^- = (\tilde n^-C \sigma(\omega)^{-1}b^{-1}\sigma(\omega)C^{-1})^-.\]
	Thus, $F(u'x)=u'\cdot F(x)$ and we conclude that $F$ is $U$-equivariant.

	Now, $(B\omega P)/L*\mcU(X)\to (B\omega P)/L*X$, $(x,v,y)\mapsto (v^{-1}x,y)$, is a $U$-morphism by Lemma \ref{lem:magic-B-equivariancy}(1). Composing it with $F\times\operatorname{id}_X$ and with the isomorphism $c*X\simeq X*c$, we get a $U$-morphism from $(B\omega P)/L*\mcU(X)$ to $\rho(\omega)*\rho(\omega^{-1})*X*c$. This induces a morphism of twisted $B$-varieties $(B\omega P)/L*\mcU(X)\to\mcU(\rho(\omega)*\rho(\omega^{-1})*X*c)$. Finally, we add the coordinate $\sigma(\omega)\tilde n^+\sigma(\omega)^{-1}\in U_{\omega^{-1}}^+\cap{}^\omega N$, where $v^{-1}x=u\sigma(\omega)n$, on which $t\in T$ acts by conjugation. The resulting morphism to $\mcU(\rho(\omega)*\rho(\omega^{-1})*X*c)*\underline{U_{\omega^{-1}}^+\cap {}^\omega N}$ is $T$-equivariant and has an algebraic inverse, as we can recover $u$, $\tilde n$, and hence $n$, $y$ and $v$, from the image.
\end{proof}

Summing up, combining Theorems \ref{thm:handle} and \ref{thm:adding-punctures} with the walk stratification of Theorem \ref{thm:walk-str}, we obtain a weak decomposition of $(\mcZ(H)^g\circ \mcZ(L_{\mcC_1})\circ \cdots \circ \mcZ(L_{\mcC_{m-1}}))(\mcU(C_m))$ for any $g\geq 0$, $\mcC_1,\ldots,\mcC_{m-1}$ nice conjugacy classes, and $C_m\in T$. Precisely, it is indexed by tuples $(\overline{\omega},p)$ where
\[\overline{\omega}=(\omega_1,\ldots, \omega_{m-1},\omega_1^{(g)},\omega_2^{(g)}, \ldots, \omega_1^{(1)}, \omega_2^{(1)})\in \Omega_G^{\mcC}:=W/W_{I_1}\times\cdots\times W/W_{I_{m-1}}\times W^{2g},\]
with $\mcC=(\mcC_1,\ldots,\mcC_m)$ and $\mcC_m$ the conjugacy class of $C_m$, and $p$ is a walk associated with
\[\beta_{\overline{\omega}} = \left[\omega_1^{(g)},\omega_2^{(g)}\right]\cdots \left[\omega_1^{(1)},\omega_2^{(1)}\right]\omega_1 \omega_1^{-1} \cdots \omega_{m-1}\omega_{m-1}^{-1}\]
where all products are taken in $\Br$ and we have fixed a reduced expression $\beta_{\overline{\omega}}=s_{\alpha_l}\cdots s_{\alpha_1}$ to take walks with respect to it. Let $\mcI_G^{\mcC}$ be the set of such tuples. The cell associated to such a tuple is
\[C_{\overline{\omega},p} = \mcU(\rho(p_{l}) * T_{\overline{\omega},p}) * \underline{U_{(\overline{\omega},p)}}\]
where $l$ is the length of $\beta_{\overline{\omega}}$,
\[U_{(\overline{\omega},p)} = {}^{p_l}\left(\prod_{k\in D_p} U_{p_k^{-1}\cdot \alpha_k}\right) * \left(\prod_{i=1}^g U_{\omega_1^{(i)}}^+ * U_{\left(\omega_2^{(i)}\right)^{-1}}^+ \right)*\left(\prod_{i=1}^{m-1}\left(U_{\omega_i^{-1}}^+\cap {}^{\omega_i}N_i\right)\right) \]
as a $T$-variety and
\[T_{\overline{\omega},p} = \left(\prod_{k\in D_p} \tau_{p_k^{-1}\cdot \alpha_k} \right) * \left(\prod_{k\in S_p} T_{p_k^{-1}\cdot \alpha_k}\right) * C_m * \left(\prod_{i=1}^{m-1} {}^{\omega_i}(C_i \tau_{\omega_i})\right) * \left(\prod_{i=1}^g T\left(\omega_1^{(i)},\omega_2^{(i)}\right)\right)\]
as a twisted $T$-variety.

\begin{rmk}
	In general, if $X=\sqcup X_i$ is a weak decomposition with $X_i\simeq Y_i$ as twisted $T$-varieties and $Y_i=\sqcup Y_{ij}$ is a weak decomposition, the preimages $X_{ij}\subset X_i$ of the $Y_{ij}$ might fail to be $B$-stable in $X$. Nonetheless, in the previous case we do obtain a weak decomposition thanks to Remarks \ref{rmk:handle-B} and \ref{rmk:puncture-B}.
\end{rmk}

\section{Cupping}\label{sec:cupping}

In this section, we will compute $\mcZ(D^\dagger)$ for the varieties $C_{\overline{\omega},p}$ appearing in the decomposition we just discussed. This amounts to describing the fibers of the structure morphism $f:T_{\overline{\omega},p}\to T$ for  $(\overline{\omega},p)\in \mcI_G^{\mcC}$. In Subsection \ref{sec:tori-fibers} we already did this for the genus $0$ case.

For $\omega_1,\omega_2\in W$, let
\[Q^\vee(\omega_1,\omega_2):=( (\operatorname{Id}- \omega_1) (\t(\Z)) + (\operatorname{Id} - \omega_2)(\t(\Z)) )\]
and $T^{\langle \omega_1,\omega_2\rangle }$ be the fixed points under both $\omega_1$ and $\omega_2$. Fix a $W$-invariant inner product on $\t$ and write $S^\perp$ for the orthogonal complement of a set $S\subset \t$.

\begin{lema}\label{lem:stab-conmutator-variety}
	Let $\omega_1,\omega_2\in W$. Then \begin{enumerate}
		\item the structure morphism $T(\omega_1,\omega_2)\to T$ is a homomorphism of tori whose induced map on cocharacter lattices has image $Q^\vee(\omega_1,\omega_2)$; in particular, the tangent algebra to its image is $Q^\vee(\omega_1,\omega_2)\otimes \C$, and
		\item the stabilizer of any point of $T(\omega_1,\omega_2)$ is $T^{\langle \omega_1,\omega_2\rangle}$ and its tangent algebra is $Q^\vee(\omega_1,\omega_2)^\perp$.
	\end{enumerate}
\end{lema}
\begin{proof}
	For $1.$, we can simply use the description of Remark \ref{rmk:handle-tori}. The first part of $2.$ follows from the very definition of the action on $T(w)$. For the tangent algebra, we note that the tangent algebra to $T^{\langle \omega_1,\omega_2\rangle}$ is $\t^{\omega_1}\cap \t^{\omega_2}$ and that $v\in \t^{\omega_i}$ if and only if
	\begin{align*}
		\langle v-\omega_iv, v' \rangle =       \langle v,v'\rangle-\langle \omega_iv,v'\rangle =
		\langle v,v'\rangle-\langle v,\omega_i^{-1}v'\rangle =
		\langle v,v'-\omega_i^{-1}v'\rangle &= 0
	\end{align*}
	for any $v'\in \t$, in other words, $v\in \img(\omega_i^{-1}-\operatorname{Id})^\perp=\img(\operatorname{Id}-\omega_i)^\perp$, as $\omega_i^{-1}-\operatorname{Id}=(\operatorname{Id}-\omega_i)\omega_i^{-1}$.
\end{proof}

\begin{lema}\label{lem:stab-handles}
	Let $g\geq 1$ and $\omega_1^{(i)},\omega_2^{(i)}\in W$ for $i=1,\ldots,g$. Then the stabilizer of any point of $T\left(\omega_1^{(1)},\omega_2^{(1)}\right)*\cdots *T\left(\omega_1^{(g)},\omega_2^{(g)}\right)$ is $\bigcap\limits_{i=1}^g T^{\langle \omega_1^{(i)},\omega_2^{(i)}\rangle}$.
\end{lema}
\begin{proof}
	We proceed by induction on $g$, the case $g=1$ being Lemma \ref{lem:stab-conmutator-variety}. Write a point as $(y,z)$ with $y\in T\left(\omega_1^{(1)},\omega_2^{(1)}\right)*\cdots *T\left(\omega_1^{(g-1)},\omega_2^{(g-1)}\right)$ and $z\in T\left(\omega_1^{(g)},\omega_2^{(g)}\right)$. By definition of the convolution, $t\in T$ stabilizes $(y,z)$ if and only if $t$ stabilizes $z$ and ${}^{w}t$ stabilizes $y$, where $w=\left[\omega_2^{(g)},\omega_1^{(g)}\right]$. As ${}^wt=t$ for $t\in T^{\langle \omega_1^{(g)},\omega_2^{(g)}\rangle}$, the claim follows from the induction hypothesis.
\end{proof}

Let $(\overline{\omega},p)\in \mcI_G^{\mcC}$. We define
\[S_{\overline{\omega}}(p):=S(p)\bigcap_{i=1}^g T^{\langle \omega_1^{(i)},\omega_2^{(i)}\rangle},\]
\[Q_{\overline{\omega}}(p):= Q(p)+\langle ((\omega_j^{(i)})^{-1}-\operatorname{Id})(\mathscr{X}): i=1,\ldots,g;j=1,2\rangle_\Z,\]
\[Z_{\overline{\omega}}(p):=(Q/Q_{\overline{\omega}}(p))^*,\]
\[Q^\vee_{\overline{\omega}}(p):= Q^\vee(p) + Q^\vee\left(\omega_1^{(g)},\omega_2^{(g)}\right) + \dots + Q^\vee\left(\omega_1^{(1)},\omega_2^{(1)}\right),\] and \[\pi_{1,\overline{\omega}}(p):=\t^{ss}(\Z)/Q^\vee_{\overline{\omega}}(p).\]
Finally, let $I_{\overline{\omega}}(p)$ be the image of $f:T_{\overline{\omega},p}\to T$. Note that if $\mcC_1\cdots \mcC_m\not \subset G^{ss}$, the character stack is empty.

\begin{thm}\label{thm:cupping}
	Let $(\overline{\omega},p)\in \mcI_G^{\mcC}$ with $\mcC_1\cdots \mcC_m \subset G^{ss}$. Then:
	\begin{enumerate}
		\item $S_{\overline{\omega}}(p)$ is the stabilizer in $T$ of any point of $T_{\overline{\omega},p}$, it contains the stabilizer of any point of $C_{\overline{\omega},p}$, and its tangent algebra is $Q^\vee_{\overline{\omega}}(p)^\perp$,

		\item the tangent space to $I_{\overline{\omega}}(p)$ is $Q^\vee_{\overline{\omega}}(p)\otimes \C$,

		\item $S_{\overline{\omega}}(p)/Z(G)\simeq Z_{\overline{\omega}}(p)$,

		\item $Z_{\overline{\omega}}(p)$ is finite if and only if $T_{\overline{\omega},p}\to T^{ss}$ is surjective, and

		\item if $T_{\overline{\omega},p}\to T^{ss}$ is surjective, $\pi_{1,\overline{\omega}}(p)$ is finite and
		\[f^{-1}(t)\simeq (\C^\times)^{2g\dim T+|S_p|-\dim T^{ss}}\times \pi_{1,\overline{\omega}}(p)\]
		for any $t\in T^{ss}$.
	\end{enumerate}
\end{thm}
\begin{proof}
	We have already proved this for genus $0$. For the general case, we apply Lemmas \ref{lem:stab-conmutator-variety} and \ref{lem:stab-handles}. The tangent space to $I_{\overline{\omega}}(p)$ is $Q^\vee_{\overline{\omega}}(p)\otimes\C$, where
	\[Q^\vee_{\overline{\omega}}(p) = Q^\vee(p) + Q^\vee\left(\omega_1^{(g)},\omega_2^{(g)}\right) + \dots + Q^\vee\left(\omega_1^{(1)},\omega_2^{(1)}\right)\]
	and, by Lemma \ref{lem:stab-handles} and the proof of Theorem \ref{thm:stabilizers}, the stabilizer of any point of $T_{\overline{\omega},p}$ is
	\[ S(p) \bigcap_{i=1}^g T^{\langle \omega_1^{(i)},\omega_2^{(i)}\rangle} \]
	whose tangent algebra is
	\[Q^\vee\left(\omega_1^{(g)},\omega_2^{(g)}\right)^\perp \bigcap \dots \bigcap Q^\vee\left(\omega_1^{(1)},\omega_2^{(1)}\right)^\perp \cap Q^\vee(p)^\perp = Q^\vee_{\overline{\omega}}(p)^\perp.\]
	This shows $1.$ and $2.$

	For $3.$, note that $t\in T^{\omega}$ if and only if $\chi(t)=1$ for any $\chi\in(\omega^{-1}-\operatorname{Id})(\mathscr{X})$, as $\chi({}^{\omega}t\,t^{-1})=((\omega^{-1}-\operatorname{Id})\chi)(t)$. Hence, writing $x=\mcE(y)$, $x\in S_{\overline{\omega}}(p)$ if and only if $\chi(y)\in\Z$ for all $\chi\in Q_{\overline{\omega}}(p)$, and we conclude as in the proof of Theorem \ref{thm:stabilizers} with $Q_{\overline{\omega}}(p)$ in place of $Q(p)$.

	The proofs of $4.$ and $5.$ are the same as for genus $0$.
\end{proof}

\begin{coro}
	Let $g,m$ be non-negative integers, $\mcC_1,\ldots,\mcC_{m-1}$ be nice conjugacy classes and $C_m\in T$. Then $(\mcZ(D^\dagger)\circ\mcZ(H)^g\circ \mcZ(L_{\mcC_1})\circ \cdots \circ \mcZ(L_{\mcC_{m-1}}))(\mcU(C_m))$ has a $B$-equivariant cell decomposition by varieties of the form $\C^i\times (\C^\times)^j$.
\end{coro}

\section{Quotient by \texorpdfstring{$T$}{T}}\label{sec:quotient-by-T}

\begin{lema}
	Let $f:T\to T'$ be an algebraic group homomorphism between affine tori with kernel $F$ and $T$ connected. Let $X$ be a $T$-variety. Then $[(X\times T')/T]\simeq [X/F]\times (T'/f(T))$ where $T$ acts on $T'$ by $f$.
\end{lema}
\begin{proof}
	We claim that $T'\simeq (T/F)\times (T'/f(T))$ as tori. To this end, we can assume that $F$ is trivial and $T'$ is connected. In this case, $f$ induces an injective morphism $df:\t(\Z)\to\t'(\Z)$. Applying the Smith normal form, we obtain a $\Z$-basis where $df$ is diagonal. If one of its diagonal entries is not $\pm1$, $f$ would not be injective. Hence, in some coordinates, $f$ looks like an inclusion $(\C^\times)^{n}\to (\C^\times)^{m}$, $(x_1,\ldots,x_n)\mapsto (x_1,\ldots,x_n,1,\ldots,1)$. The claim follows.

	Going back to the statement, we have
	$X\times T'\simeq X\times (T/F)\times (T'/f(T))$
	and therefore the claim becomes
	$[(X\times (T/F))/T] \simeq [X/F]$. Now, $X\times (T/F)$ is isomorphic as a $T$-variety to the induction from $F$ to $T$ of the $F$-action on $X$.
\end{proof}

For a variety $f:X\to G$ over $G$ and $t\in T$, set $X(t)$ to be the preimage of $t$ under $f$. In addition, given $\overline{\omega}\in \Omega_G^{\mcC}$, and chosen $N_1,\ldots,N_{m-1}$, let
\[d(\overline{\omega})= \sum_{j=1}^g \left(\dim U_{\omega_1^{(j)}}^+   + \dim U_{\omega_2^{(j)}}^+ \right) + \sum_{j=1}^{m-1} \dim \left(N_j\cap U_{\omega_j}^+\right).\]

\begin{coro}\label{cor:cells-over-one}
	For a genus $g$, let $(\overline{\omega},p)\in \mcI_G^{\mcC}$ with $p$ finishing at $e$. Assume that $\mcC$ is generic. Then $[C_{\overline{\omega},p}(1)/ (T/Z(G))]$ is either empty or isomorphic to \[\pi_{1,\overline{\omega}}(p)\times(\C^\times)^{2g\dim T+|S_p|-2\dim T^{ss}}\times [\C^{d(\overline{\omega})+|U_p|}/Z_{\overline{\omega}}(p)],\]
	where $Z_{\overline{\omega}}(p)$ acts in a linear way on $\C^{d(\overline{\omega})+|U_p|}$.
\end{coro}
\begin{proof}
	As $\mcC$ is generic, $T_{\overline{\omega},p}\to T^{ss}$ must be surjective for $C_{\overline{\omega},p}(1)$ to be non-empty. It follows that $Z_{\overline{\omega}}(p)$ and $\pi_{1,\overline{\omega}}(p)$ are finite and that $C_{\overline{\omega},p}(1)\simeq \C^{d(\overline{\omega})+|U_p|}\times (\C^\times)^{2g\dim T+|S_p|-\dim T^{ss}}\times\pi_{1,\overline{\omega}}(p)$.

	The action of $T$ factorizes through $T'=T/Z(G)$, which is connected as $G$ is. Now, $T'$ acts on $(\C^\times)^{2g\dim T+|S_p|-\dim T^{ss}}$ by a group homomorphism whose kernel is exactly $Z_{\overline{\omega}}(p)$. Hence, the result follows from the previous lemma.
\end{proof}

\begin{Def}\label{def:generic}
	More precisely, the genericity condition is the following. We define, for any subset $S\subset( \Phi\cup \{0\})^2$,
	\[\t_S:=\langle h_\alpha-h_\beta:(\alpha,\beta)\in S\rangle_\C\]
	where $h_0:=0$. Let $\mcT_S\subset T$ be the union of all closed subgroups whose tangent algebras are $\t_S$. We say that $t\in T$ is generic if $t\not\in \mcT_S$ for all $S$ such that $\t_S\neq \t^{ss}$. We note that there exist generic elements.
\end{Def}

\begin{Def}
	We say that a tuple of semisimple conjugacy classes $\mcC_1,\ldots,\mcC_m$ of $G$ is generic if for some (any) maximal torus $T$ and $C_i\in \mcC_i\cap T$, and any $w_1,\ldots,w_m\in W$, ${}^{w_1}C_1\cdots {}^{w_m}C_m$ is generic in the sense of Definition \ref{def:generic}.
\end{Def}

\begin{thm}\label{thm:cell-decomposition}
	Let $G$ be a connected and reductive algebraic group, $g\geq 0$ and $m\geq 1$ be integers, and $\mcC_1,\ldots,\mcC_m$ be a generic tuple of nice semisimple conjugacy classes of $G$. Assume that $\mcC_m$ is regular. Then $\tBetti[G]$ has a cell decomposition by finite quotients of varieties of the form $(\C^\times)^{d-2i}\times \C^{i+r}$, where $d$ is the dimension of $\abbreviatedBetti[G]$ and $r$ is the dimension of any maximal unipotent subgroup of $G$.
\end{thm}
\begin{proof}
	Fix a maximal torus $T$ and a Borel subgroup of $G$ and let $C_m\in \mcC_m\cap T$. We recall that
	\[\tBetti[G] \simeq [(\mcZ(D^\dagger)\circ \mcZ(H)^g\circ \mcZ(L_{\mcC_1})\circ \cdots \circ \mcZ(L_{\mcC_{m-1}}))(\mcU(C_m)) / (T/Z(G))]\]
	and that we can move the untwisting operation to the left, transforming the handles and the punctured cylinders into a disjoint union of fiber bundles over $\rho(\beta_{\overline{\omega}})$ for $\overline{\omega}\in \Omega_G^{\mcC}$. Hence, applying the walk stratification and noticing that, for any $e$-twisted $X\in \TwBVarG$ and $t\in T$,
	\[ \{(u,x)\in\mcU(X):uf(x) =t \} = \{x\in X:f(x)\in tU\},\]
	we get a decomposition by finite quotients of varieties of the form $(\C^\times)^i\times \C^j$ as long as all torus elements appearing are generic. This last requirement follows from the definition of genericity, as the subgroup of $G$ generated by $\sigma(\alpha)$, $\alpha\in \Delta^+$, is known to be finite \cite{Tits}.

	We must check the claimed relation between the multiplicities of $\C$ and $\C^\times$ on each cell. Pick $C_j\in \mcC_j\cap T$ nice and let $N_j$ be the unipotent radical of its associated parabolic. Fix $\overline{\omega}\in \Omega_G^{\mcC}$. Then $d(\overline{\omega})$ is
	\begin{align*}
		\sum_{j=1}^g &\left(\dim U_{\omega_1^{(j)}}^+   + \dim U_{\omega_2^{(j)}}^+ \right) + \sum_{j=1}^{m-1} \dim \left(N_j\cap U_{\omega_j}^+\right) = \\
		&= \sum_{j=1}^g\left(\dim G-\dim T- \ell\left(\omega_1^{(j)}\right)-\ell\left(\omega_2^{(j)}\right)\right)+\sum_{j=1}^{m-1}\left(\frac{1}{2}(\dim G-\dim Z(C_j))-\ell(\omega_j)\right).
	\end{align*}
	On the other hand, by Corollary \ref{cor:cells-over-one}, every connected component of $[C_{\overline{\omega},p}(1)/(T/Z(G))]$, for a walk $p$ associated with $\beta_{\overline{\omega}}$ and finishing at $e$, is of the form
	\[(\C^\times)^{2g\dim T+|S_p|-2\dim T^{ss}}\times [\C^{d(\overline{\omega})+|U_p|}/Z_{\overline{\omega}}(p)].\]
	Now $|S_p|=\ell(\beta_{\overline{\omega}})-2|U_p|$. Hence,
	\[|S_p|=2\left(\sum_{j=1}^{m-1} \ell(\omega_j)+\sum_{j=1}^g\left(\ell\left(\omega_1^{(j)}\right)+\ell\left(\omega_2^{(j)}\right)\right)-|U_p|\right)\]
	and, in particular, $|S_p|$ is even. Let
	\[i=\frac{1}{2}\left((2g-1)(\dim G -\dim T)+\sum_{j=1}^{m-1}(\dim G-\dim Z(C_j))\right)-\frac{1}{2}|S_p|\]
	which is an integer. We note that $\dim Z(C_m)=\dim T$. Therefore, the final cell has $r + i$ copies of $\C$ and
	\[(2g-2)\dim G +\sum_{j=1}^{m}(\dim G-\dim Z(C_j))+2\dim Z(G)-2i\]
	copies of $\C^\times$ as wanted.
\end{proof}

\begin{coro}
	Let $G$ be a connected and reductive algebraic group, $g\geq 0$ and $m\geq 1$ be integers, and $\mcC_1,\ldots,\mcC_m$ be a generic tuple of nice semisimple conjugacy classes of $G$. Assume that $\mcC_m$ is regular. Then $\tBetti[G]$ is a Deligne--Mumford stack.
\end{coro}

\begin{ex}
	Example \ref{ex:bad-cell} does not show up on any character stack. A slightly more complicated one does. We continue the notation of Example \ref{ex:bad-cell}. The elements of the Weyl group, written in minimal form, are:
	\[e,s_{\alpha_1},s_{\alpha_2},s_{\alpha_2}s_{\alpha_1},s_{\alpha_1}s_{\alpha_2},s_{\alpha_2}s_{\alpha_1}s_{\alpha_2},s_{\alpha_1}s_{\alpha_2}s_{\alpha_1}\text{ and }s_{\alpha_1}s_{\alpha_2}s_{\alpha_1}s_{\alpha_2}.\]
	We take $g=0$, $m=5$, all conjugacy classes regular and semisimple, $\omega_1=\omega_4=s_{\alpha_1}$ and $\omega_2=\omega_3=s_{\alpha_1}s_{\alpha_2}s_{\alpha_1}$. The following is a walk for the associated braid $\beta_{\overline{\omega}}$:
	\begin{table}[H]
		\centering
		\begin{tabular}{c|cccccccc}
			$k$ & $1$ & $2$ & $3$ & $4$ & $5$ & $6$ & $7$ & $8$  \\
			\hline
			root & $\alpha_1$ & $\alpha_1$ & $\alpha_1$ & $\alpha_2$ & $\alpha_1$ & $\alpha_1$ & $\alpha_2$ & $\alpha_1$ \\
			step & up & stay & down & up & up & stay & up & up \\
			$p_k$ & $s_{\alpha_1}$ & $s_{\alpha_1}$ & $e$ & $s_{\alpha_2}$ & $s_{\alpha_1}s_{\alpha_2}$ & $s_{\alpha_1}s_{\alpha_2}$ & $s_{\alpha_2}s_{\alpha_1}s_{\alpha_2}$ & $s_{\alpha_1}s_{\alpha_2}s_{\alpha_1}s_{\alpha_2}$ \\
			$\pi_k$ & $s_{\alpha_1}$ & $e$ & $s_{\alpha_1}$ & $s_{\alpha_2}s_{\alpha_1}$ & $s_{\alpha_1}s_{\alpha_2}s_{\alpha_1}$ & $s_{\alpha_2}s_{\alpha_1}$ & $s_{\alpha_1}$ & $e$
		\end{tabular}

		\medskip

		\begin{tabular}{c|cccccccc}
			$k$ & $9$ & $10$ & $11$ & $12$ & $13$ & $14$ & $15$ & $16$\\
			\hline
			root & $\alpha_1$ & $\alpha_2$ & $\alpha_1$ & $\alpha_1$ & $\alpha_2$ & $\alpha_1$ & $\alpha_1$ & $\alpha_1$ \\
			step & down & down & stay & down & down & up & stay & down \\
			$p_k$ & $s_{\alpha_2}s_{\alpha_1}s_{\alpha_2}$ & $s_{\alpha_1}s_{\alpha_2}$ & $s_{\alpha_1}s_{\alpha_2}$ & $s_{\alpha_2}$ & $e$ & $s_{\alpha_1}$ & $s_{\alpha_1}$ & $e$ \\
			$\pi_k$ & $s_{\alpha_1}$ & $s_{\alpha_2}s_{\alpha_1}$ & $s_{\alpha_1}s_{\alpha_2}s_{\alpha_1}$ & $s_{\alpha_2}s_{\alpha_1}$ & $s_{\alpha_1}$ & $e$ & $s_{\alpha_1}$ & $e$
		\end{tabular}
		\caption{Another walk with $I_{\overline{\omega}}(p)=T$ and non-trivial center $Z_{\overline{\omega}}(p)$ on $\Sp_4(\C)$.}
	\end{table}
	For this walk, $\pi_{1,\overline{\omega}}(p)$ is trivial and $Z_{\overline{\omega}}(p)=\Z/2\Z$ as in Example \ref{ex:bad-cell}.
\end{ex}

\begin{thm}\label{thm:no-quotients-typeA-unipvar}
	Let $G$ be a connected reductive group of type $A$ and $(\overline{\omega},p)\in \mcI_G^{\mcC}$ with $T_{\overline{\omega},p}\to T^{ss}$ surjective. If $g=0$, then $\pi_{1,\overline{\omega}}(p)=\pi_1(G^{ss})$ and $Z_{\overline{\omega}}(p)$ is trivial. If $G^{ss}$ is simply connected, then $Z_{\overline{\omega}}(p)$ is trivial.
\end{thm}
\begin{proof}
	Let us first show that, if $G^{ss}$ is simply connected, $(w-\operatorname{id})(\mathscr{X})$ is generated over $\Z$ by a subset of $\Phi$ for any $w\in W$. As $W$ acts trivially on the center of $G$, we can restrict to characters of $T^{ss}$. Hence, as $G^{ss}$ is simply connected, we can assume that $G=\SL_n$. In this case,
	\[\t=\{x\in \C^{n}:x_1+\ldots+x_n=0\},\]
	\[\mathscr{X} = \Z^n/ (1,\ldots,1)\Z ,\]
	\[\Phi=\{e_i-e_j:1\leq i\neq j\leq n\},\]
	and $W=S_n$ acts by permutation of the indices. The claim follows as $(w-\operatorname{id})(e_i)=e_{w(i)}-e_i$.

	Going back to the statement, what we just showed implies that $Q_{\overline{\omega}}(p)$ is generated over $\Z$ by a subset of roots and, if $g=0$, that $Q^\vee_{\overline{\omega}}(p)=Q^\vee(p)$ is generated over $\Z$ by a subset of coroots. Now, both lattices have full rank as $T_{\overline{\omega},p}\to T^{ss}$ is surjective. Hence, it is enough to show that a subset $S\subset \Phi$ generates $\langle \Phi\rangle_\Z$ if and only if it generates $\t$ over $\C$. Again, we can assume that $G$ is simple. In this case,
	\[\langle\Phi\rangle_\Z=\{x\in \Z^{n}:x_1+\ldots+x_n=0\}.\]
	Given a subset $S\subset \Phi$, we consider the graph with vertices labeled from $1$ to $n$ and an edge between $i$ and $j$ if and only if $\pm(e_i-e_j)\in S$. We note that $S$ generates $\t$ over $\C$ or $\langle \Phi\rangle_\Z$ over $\Z$ if and only if this graph is connected.
\end{proof}

\begin{ex}
	The hypotheses of Theorem \ref{thm:no-quotients-typeA-unipvar} cannot be dropped. Take $G$ to be $\SL_2$ or $\PGL_2$, $g=1$, $m=1$ and $\overline{\omega}=(s,e)$, where $s$ is the non-trivial element of $W$. We take as $p$ the walk that simply goes up and then down, so that $S_p=\emptyset$. In this case, $Q_{\overline{\omega}}(p)=(s-\operatorname{Id})(\mathscr{X})$ and $Q^\vee_{\overline{\omega}}(p)=(\operatorname{Id}-s)(\t(\Z))$. So, for $\PGL_2$, we get $Z_{\overline{\omega}}(p)=\Z/2\Z$. Meanwhile, for $\SL_2$, we get $\pi_{1,\overline{\omega}}(p)=\Z/2\Z\neq \pi_1(\SL_2)$. Finally, $T_{\overline{\omega},p}\to T$ is surjective, as it is $(t_1,t_2)\mapsto c\,t_2^2$ for a constant $c\in T$.
\end{ex}

\begin{coro}
	Let $G$ be a connected and reductive algebraic group of type $A$ with $G^{ss}$ simply connected, $g\geq 0$ and $m\geq 1$ be integers, and $\mcC_1,\ldots,\mcC_m$ be a generic tuple of nice semisimple conjugacy classes of $G$. Assume that $\mcC_m$ is regular. Then $\tBetti[G]$ has a cell decomposition by varieties of the form $(\C^\times)^{d-2i}\times \C^{i+r}$, where $d$ is the dimension of $\abbreviatedBetti[G]$ and $r$ is the dimension of any maximal unipotent subgroup of $G$.
\end{coro}
\begin{proof}
	By Theorem \ref{thm:no-quotients-typeA-unipvar}, we know that no non-trivial finite quotients appear in the stratification of Theorem \ref{thm:cell-decomposition}.
\end{proof}

\begin{rmk}
	All semisimple conjugacy classes of $\GL_n$ and $\SL_n$ are nice.
\end{rmk}

\section{Quotient by \texorpdfstring{$B$}{B}}\label{sec:bundle-structure}

In this section, we explore the relation between $\tBetti[G]$ and $\abbreviatedBetti[G]$. We fix as before a connected and reductive algebraic group $G$, a Borel subgroup $B$ with unipotent radical $U$, and a maximal torus $T\subset B$.

\begin{lema}\label{lem:B-quotient}
	Let $\mcC_1,\ldots, \mcC_m$ be conjugacy classes of $G$. Assume that $\mcC_m$ is regular semisimple with connected centralizer and pick $C_m\in \mcC_m\cap T$. Then
	\[ \abbreviatedBetti[G] \simeq [(\mcZ(D^\dagger)\circ \mcZ(H)^g\circ \mcZ(L_{\mcC_1})\circ \cdots \circ \mcZ(L_{\mcC_{m-1}}))(\mcU(C_m)) / (B/Z(G))]\]
	and there is a morphism $\tBetti[G]\to \abbreviatedBetti[G]$.
\end{lema}
\begin{proof}
	Since $C_m$ is regular, the map $U\to U$, $u\mapsto u^{-1}C_muC_m^{-1}$, is an isomorphism of varieties. Applying its inverse, we get an isomorphism between $\tBetti[G]$ and
	\[[\{(x,y,z,u)\in G^{2g}\times \mcC_1\times \dots\times \mcC_{m-1}\times U: [x_1,y_1]\cdots[x_g,y_g]z_1\cdots z_{m-1}u^{-1}C_mu=1 \}/ (T/Z(G))].\]
	We make the change of variables $(x,y,z,u)\mapsto (uxu^{-1},uyu^{-1},uzu^{-1},u)$ to get
	\[\tBetti[G]\simeq [\{(x,y,z,u)\in G^{2g}\times \mcC_1\times \dots\times \mcC_{m-1}\times U: [x_1,y_1]\cdots[x_g,y_g]z_1\cdots z_{m-1}C_m=1 \}/ (T/Z(G))]\]
	where the $T$-action on $U$ is still the conjugation action. From the previous description, we have the ``forget the $u$-variable'' morphism $\pi:\tBetti[G]\to \abbreviatedBetti[G]$.

	Under the inverse of $u\mapsto u^{-1}C_muC_m^{-1}$, the $B$-action changes to $(bxb^{-1},byb^{-1},bzb^{-1}, tub^{-1})$, where $t$ is the image of $b$ under $B\to B/U\simeq T$. After the conjugation by $u$, the $B$-action becomes $(txt^{-1},tyt^{-1},tzt^{-1},tub^{-1})$. Hence, if we first quotient by $U$, we kill the $u$ variable and we get
	\[\{(x,y,z)\in G^{2g}\times \mcC_1\times \dots\times \mcC_{m-1}: [x_1,y_1]\cdots [x_g,y_g]z_1\cdots z_{m-1}C_m=1\}\]
	with its $T$-action, whose quotient is $\abbreviatedBetti[G]$.
\end{proof}

\begin{coro}\label{cor:cells-pullback}
	Let $G$ be a connected and reductive algebraic group, $g\geq 0$ and $m\geq 1$ be integers, and $\mcC_1,\ldots,\mcC_m$ be a generic tuple of nice semisimple conjugacy classes of $G$. Assume that $\mcC_m$ is regular. Then $\abbreviatedBetti[G]$ is a Deligne--Mumford stack and the cell decomposition of $\tBetti[G]$ is the pullback under $\tBetti[G] \to \abbreviatedBetti[G]$ of a decomposition of $\abbreviatedBetti[G]$.
\end{coro}
\begin{proof}
	The cells of $(\mcZ(D^\dagger)\circ \mcZ(H)^g\circ \mcZ(L_{\mcC_1})\circ \cdots \circ \mcZ(L_{\mcC_{m-1}}))(\mcU(C_m))$ are $B$-invariant. Since the quotient by $U$ is a $U$-fiber bundle, we get a $T$-invariant cell decomposition of
	$(\mcZ(D^\dagger)\circ \mcZ(H)^g\circ \mcZ(L_{\mcC_1})\circ \cdots \circ \mcZ(L_{\mcC_{m-1}}))(C_m)$.
	This induces the desired cell decomposition on $\abbreviatedBetti[G]$.
\end{proof}

\begin{lema}\label{lem:bundle-structure}
	Let $G$ be of type $A$ with $G^{ss}$ simply connected, and $\mcC_1,\ldots,\mcC_m$ be a generic tuple of nice semisimple conjugacy classes of $G$ with $\mcC_m$ regular. Then $\tBetti[G]\to\abbreviatedBetti[G]$ is Zariski locally trivial with fiber $U$.
\end{lema}
\begin{proof}
	By Theorems \ref{thm:no-quotients-typeA-unipvar} and \ref{thm:cupping}, the stabilizers in $T/Z(G)$ of the points of the non-empty cells are trivial. Hence, the action of $T/Z(G)$ on
	\[R=\{(x,y,z)\in G^{2g}\times \mcC_1\times \dots\times \mcC_{m-1}: [x_1,y_1]\cdots [x_g,y_g]z_1\cdots z_{m-1}C_m=1\}\]
	is free. Hence, Luna's \'etale slice theorem \cite{Luna} ensures that the quotient morphism $R\to \abbreviatedBetti[G]$ is a principal $T/Z(G)$-bundle, locally trivial in the \'etale topology. Moreover, since $T/Z(G)$ is a torus, it is a special group and therefore $R\to \abbreviatedBetti[G]$ is locally trivial in the Zariski topology.  Thus, $\pi$ is Zariski locally trivial with fiber $U$.
\end{proof}

\part{Motivic Mirror Symmetry}

\section{Grothendieck ring of stacks}

We follow \cite{Ekedahl}. By a stack, we will mean an algebraic stack of finite type over $\C$ with affine stabilizers at every closed point. Let $K_0(\Stk_\C)$ be the Grothendieck ring of stacks and $q\in K_0(\Stk_\C)$ be the class of the affine line. Remark that $q$ and $q^n-1$, $n\geq 1$, are invertible in $K_0(\Stk_\C)$, and that for any vector bundle $\mcX\to \mcY$ of rank $n$, $[\mcX]=q^n[\mcY]$. In addition, for any special group $H$ and $H$-stack $\mcX$, $[\mcX/H] = [\mcX][H]^{-1}$. In particular, this applies to tori.

\begin{lema}\label{cor:motivically-a-vector-bundle}
	Let $G$ be a connected and reductive algebraic group, $g\geq 0$ and $m\geq 1$ be integers, and $\mcC_1,\ldots,\mcC_m$ be a generic tuple of nice semisimple conjugacy classes of $G$. Assume that $\mcC_m$ is regular. Then \[[\abbreviatedBetti[G]] = [\tBetti[G]] q^{-\frac{1}{2}(\dim G- \rk G)}\]
	in $K_0(\Stk_\C)$.
\end{lema}
\begin{proof}
	This holds as $\tBetti\to \abbreviatedBetti[G]$ is a vector bundle.
\end{proof}

\section{Stringy motive}

Given a variety $X$, we say that its class $[X]$ in $K_0(\Stk_\C)$ is its (naive) motive. We want a stringy version of this class along the lines of \cite{BD}, \cite{CR}, \cite{GWZ}, etc.

Let us recall how the stringy $E$-polynomial of a global finite quotient stack $\mcX=[X/\Gamma]$ is defined, where $\Gamma$ is a finite group and $X$ is a smooth $\Gamma$-variety. Given a point $x\in X$ fixed by $\gamma\in \Gamma$, the tangent space $T_xX$ inherits an action of $\gamma$. Its eigenvalues can be written uniquely as $e^{2\pi i c_i}$ with $0\leq c_i<1$. The fermionic shift of $\gamma$ at $x$ is $F(\gamma,x) := \sum_i c_i$. Varying $x$, we obtain a locally constant function on $X^\gamma$. Furthermore, if $C(\gamma)$ is the centralizer of $\gamma$, this function descends to $[X^\gamma / C(\gamma)]$. The stringy $E$-polynomial is defined as
\[ E_{st}(\mcX;u,v) : = \sum_{\gamma}\sum_{\mcY} E(\mcY;u,v) (uv)^{F(\gamma, \mcY)} \]
where $\gamma$ runs over a set of representatives of all conjugacy classes of $\Gamma$ and $\mcY$ runs over all connected components of $[X^\gamma/C(\gamma)]$. It turns out that this definition does not depend on the choice of the presentation $[X/\Gamma]$. In addition, if $X$ admits a $\Gamma$-invariant symplectic form, $F(\gamma,\mcY)$ is an integer and, more precisely, it is $\frac{1}{2}(\dim\mcX-\dim\mcY)$.

A slightly more intrinsic definition is given by the inertia stack $I\mcX$. It is defined as $\mcX \times_{\mcX\times\mcX} \mcX$. When $\mcX$ is a finite quotient $[X/\Gamma]$, there is an equivalence
\[I\mcX \simeq \bigsqcup_{\gamma} [X^\gamma /C(\gamma)],\]
where $\gamma$ runs over a set of representatives of the conjugacy classes of $\Gamma$, and the fermionic shifts assemble into a locally constant function $F: I\mcX(\C)_{\text{iso}} \to \Q$. Therefore,
\[E_{st}(\mcX;u,v) = \sum_{\mcY \in \pi_0(I\mcX)} E(\mcY;u,v)(uv)^{F(\mcY)}.\]

We take the last formula as the definition of the stringy motive. Let $\mcX$ be a Deligne--Mumford stack and let $F:\pi_0(I\mcX)\to \Z$ be a function. We define its $F$-stringy motive as
\[ [\mcX,F]_{st} := \sum_{\mcY \in \pi_0(I\mcX)} [\mcY] q^{F(\mcY)} \in K_0(\Stk_\C).\]
When $F(\mcY)=\frac{1}{2}(\dim \mcX-\dim \mcY)$ for all $\mcY$, we simply write $[\mcX]_{st}$ for $[\mcX,F]_{st}$ and we call $[\mcX]_{st}$ the stringy motive of $\mcX$. By the previous discussion, if $\mcX$ can be presented as a finite quotient of a smooth variety admitting an invariant symplectic form, $E([\mcX]_{st})=E_{st}(\mcX)$.

\begin{lema}\label{lem:stringy-motive-and-cell-decomposition}
	Let $\mcX = \bigsqcup \mcX_i$ be a finite cell decomposition. Then
	\[ [\mcX,F]_{st} = \sum_i [\mcX_i,F_i]_{st}\]
	where, for $\mcY\in\pi_0(I\mcX_i)$, $F_i(\mcY):=F(\hat \mcY)$ for the unique $\hat \mcY \in \pi_0(I\mcX)$ intersecting $\mcY$ non-trivially.
\end{lema}

\section{Statement}\label{sec:statement}

From now on, we will fix a connected and semisimple algebraic group $G$, integers $g\geq 0$ and $m\geq 1$, and generic nice semisimple conjugacy classes $\mcC_1,\ldots,\mcC_m$ with $\mcC_m$ regular. We also fix a Borel subgroup $B$ with unipotent radical $U$, a maximal torus $T\subset B$, and $C_m\in\mcC_m\cap T$. We recall that, in this case, $\abbreviatedBetti[G]$ and $\tBetti[G]$ are Deligne--Mumford stacks.

There is an action of $Z(G)^{2g}$ on
\[\{(x,y,z,u)\in G^{2g}\times \mcC_1\times\cdots\times\mcC_{m-1}\times U: [x_1,y_1]\cdots [x_g,y_g]z_1\cdots z_{m-1}uC_m =1\}\]
given by left multiplication on $(x,y)$. This action commutes with the $B$-action and therefore defines actions on $\tBetti[G]$ and $\abbreviatedBetti[G]$. For a subgroup $F\subset Z(G)$, we define the $G/F$-twisted character stack as the quotient stack
\[\abbreviatedBetti[G/F]:=[\abbreviatedBetti[G]/ F^{2g}].\]
It is a Deligne--Mumford stack under our running assumptions.

Let ${}^L\tilde{G}$ be the universal covering of the Langlands dual group of $G$. For any subgroup $F\subset Z(G)$, $(Z(G)/F)^*=\operatorname{Hom} (Z(G)/F,\C^\times)$ is naturally a subgroup of $Z({}^L\tilde{G})$. We also identify each simple root of $G$ with its associated coroot, which is a simple root of ${}^L\tilde{G}$. We say that a nice semisimple conjugacy class $\mcC$ of $G$ is dual to a nice semisimple conjugacy class $\mcC^\vee$ of ${}^L\tilde{G}$ if the subsets $I$ and $I^\vee$ associated with $\mcC$ and $\mcC^\vee$ as in Section \ref{sec:adding-punctures} can be chosen so that they are identified.

\begin{thm}\label{conjMotivic}
	Let $G$ be a connected, simply connected, and semisimple complex algebraic group. Let $F$ be a subgroup of its center $Z(G)$ and ${}^L\tilde{G}$ be the universal covering of its Langlands dual group. Then
	\[ [\abbreviatedBetti[G/F]]_{st}=[\mcM_{\mathrm{B}}^{\mcC^\vee}({}^L\tilde{G}/(Z(G)/F)^*)]_{st}\]
	for any $g\geq 0$ and generic tuples $\mcC$ and $\mcC^\vee$ of nice semisimple conjugacy classes of $G$ and ${}^L\tilde{G}$, respectively, that are dual to each other, with $\mcC_m$ regular.
\end{thm}

Both sides of the equality can be written in terms of root theory by means of the cell decomposition of the previous part. This allows us to describe all stabilizers and to use the Weil pairing to induce a duality between walks that interchanges their stabilizers and fundamental groups. This leads to a refined version of the theorem which is as follows.

It is straightforward to check that the $Z(G)^{2g}$-action preserves the cell decomposition and, thus, we have a $(T/Z(G))\times Z(G)^{2g}$-action on each cell $C_{\overline{\omega},p}(1)/U$, $(\overline{\omega},p)\in \mcI_G^{\mcC}$. Assume that $p$ finishes at $e$ and $T_{\overline{\omega},p}\to T$ is surjective. Let $\tilde Z_{\overline{\omega}}(p)\subset (T/Z(G))\times Z(G)^{2g}$ be the subset of those elements that fix some point of $C_{\overline{\omega},p}(1)/U$. Finally, for $\gamma\in \tilde Z_{\overline{\omega}}(p)$, let
\[ C_{\overline{\omega},p}(1)^{[\gamma]} := \bigcup_{u\in U} C_{\overline{\omega},p}(1)^{u\gamma}\]
or, equivalently, the preimage of $(C_{\overline{\omega},p}(1)/U)^\gamma$ under the $U$-quotient.

\begin{thm}\label{conjMotivic2}
	Let $G$ be a connected and semisimple complex algebraic group. Let $(\overline{\omega},p)\in \mcI_G^{\mcC}$ with $p$ finishing at $e$ and $T_{\overline{\omega},p}\to T$ surjective. Then
	\begin{align*}
		[C_{\overline{\omega},p}(1)^{[\gamma]}] q^{F(\gamma)} &= [C_{\overline{\omega},p}(1) ]
	\end{align*}
	for any $\gamma \in \tilde Z_{\overline{\omega}}(p)$, where $F(\gamma)$ is the fermionic shift on $\abbreviatedBetti[G/Z(G)]$ at the points of $(C_{\overline{\omega},p}(1)/U)^\gamma$.
\end{thm}

In the next three sections we will prove this theorem. The idea is that there is a $U$-slice on $\tBetti$ that induces a $U$-slice on each cell and this implies that fixed points differ from the original cell by a power of $q$. This power can be combinatorially computed in terms of the walk. Then a direct computation of the fermionic shift shows that they agree. In the last section of the article we will work out the details of the aforementioned duality and show that it implies Theorem \ref{conjMotivic}.

\section{The \texorpdfstring{$U$}{U}-slice}\label{sec:U-slice}

Let us consider the representation varieties
\[R :=  (\mcZ(D^\dagger)\circ \mcZ(H)^g\circ \mcZ(L_{\mcC_1})\circ \cdots \circ \mcZ(L_{\mcC_{m-1}}))(C_m)\]
and
\[\tilde R := (\mcZ(D^\dagger)\circ \mcZ(H)^g\circ \mcZ(L_{\mcC_1})\circ \cdots \circ \mcZ(L_{\mcC_{m-1}}))(\mcU(C_m))\]
so that $\abbreviatedBetti[G]= [R/(T/Z(G))]$ and $\tBetti[G]= [\tilde R/(T/Z(G))]$. Recall from the proof of Lemma \ref{lem:B-quotient} that $\Psi:\tilde R\to R\times U$, $(x,y,z,u)\mapsto((u'xu'^{-1},u'yu'^{-1},u'zu'^{-1}),u')$, where $u'\in U$ is the unique element with $u=u'^{-1}C_mu'C_m^{-1}$, is an isomorphism. Moreover, $\Psi$ is $B$-equivariant. On $\tilde R$ the action is the standard one. Meanwhile, for $(r,u)\in R\times U$, $t\in T$ and $u'\in U$, $(tu')\cdot (r,u) = (t\cdot r,tuu'^{-1}t^{-1})$, where the $T$-action on $R$ comes from the convolution structure. In addition, $\Psi$ is $Z(G)^{2g}$-equivariant, where the action on $R$ is given by left multiplication on $(x,y)$ and is trivial on $U$. Thus, we have the following.

\begin{lema}
	Let $X\subset \tilde R$ be a $U$-invariant subvariety. Then $\Psi(X) = Y\times U$ for a subvariety $Y\subset R$. In particular, $X/U\simeq Y$. Moreover, if $X$ is $B\times Z(G)^{2g}$-invariant, $Y$ is $T\times Z(G)^{2g}$-invariant.
\end{lema}
\begin{proof}
	This follows by noticing that if $\Psi(x)= (y,u)$, then $\Psi(u\cdot x) = (y,e)$.
\end{proof}

\begin{lema}\label{lem:U-slice-trick}
	Let $X\subset \tilde R$ be a $B\times Z(G)^{2g}$-invariant subvariety and $\gamma = (t, \zeta) \in T\times Z(G)^{2g}$. Then
	\begin{enumerate}
		\item $\Psi(X^\gamma) = Y^\gamma \times U^t$,
		\item $\Psi(\bigcup\limits_{u\in U} X^{u\gamma}) = Y^\gamma \times U$, and
		\item $\left[\bigcup\limits_{u\in U} X^{u\gamma}\right] = q^{d(t)} [X^\gamma]$ where $d(t)$ is the number of positive roots that do not vanish at $t$.
	\end{enumerate}
\end{lema}
\begin{proof}
	 We use the previous lemma. The first item follows from the action being diagonal. For the second one, note that $\bigcup\limits_{u\in U} X^{u\gamma} = \{ x\in X: \gamma \cdot x\in U\cdot x\}$ and $U\cdot (y,u) = \{y\}\times U$. Finally, the third one follows from
	 \[ \left[\bigcup\limits_{u\in U} X^{u\gamma}\right]=[U][Y^\gamma] = [U][U^t]^{-1}[X^\gamma]\]
	 and $U^t$ being isomorphic to the direct product of the root subgroups on which $t$ acts trivially.
\end{proof}

In particular, this means that Theorem \ref{conjMotivic2} is equivalent to
\[ 		[C_{\overline{\omega},p}(1)^\gamma] q^{F(\gamma)+d(t)} = [C_{\overline{\omega},p}(1) ] \]
for all $(\overline{\omega},p)\in \mcI_G^{\mcC}$ with $p$ finishing at $e$ and $T_{\overline{\omega},p}\to T$ surjective, and $\gamma=(t,\zeta) \in \tilde Z_{\overline{\omega}}(p)$.

\section{Fermionic shift}

Let $L_i=Z_G(C_i)$ be the Levi subgroups of Section \ref{sec:adding-punctures}, $\mathfrak{l}_i$ their tangent algebras and, for $x\in R$, let $\mathrm{Stab}_x\subset (T/Z(G))\times Z(G)^{2g}$ be the stabilizer of $x$.

\begin{lema}
	Let $x\in R$, and let $z_i\in\mcC_i$, $i=1,\ldots,m-1$, be its puncture coordinates, and write $z_i=g_iC_ig_i^{-1}$ with $g_i\in G$. Then $R$ is smooth at $x$ and there is a $\mathrm{Stab}_x$-equivariant isomorphism
	\[T_xR\oplus\g\simeq\g^{2g}\oplus\g/\operatorname{Ad}(g_1)(\mathfrak{l}_1)\oplus\dots\oplus\g/\operatorname{Ad}(g_{m-1})(\mathfrak{l}_{m-1}),\]
	where $(t,\zeta)\in\mathrm{Stab}_x$ acts by $\operatorname{Ad}(t)$ on $\g$ and on every summand of the right-hand side.
\end{lema}
\begin{proof}
	This follows from the regular value theorem, as $G^{2g}\times \mcC_1\times\dots\times \mcC_{m-1}\to G$, $(x',y',z')\mapsto [x'_1,y'_1]\cdots [x'_g,y'_g]z'_1\cdots z'_{m-1}C_m$, has the identity as a regular value. The last claim follows as in \cite[Theorem 2.2.5]{HRV}.
\end{proof}

For $t\in T$, $1\leq i\leq m-1$ and $\omega_i\in W/W_{I_i}$, let $d_{\omega_i}(t)$ be the number of positive roots $\alpha$ whose root subgroup $U_\alpha$ is not contained in $L_i$ and such that $(\omega_i\cdot\alpha)(t)\neq 1$.

\begin{coro}\label{cor:fermionic-shift}
	Let $\gamma=(t,\zeta)\in (T/Z(G))\times Z(G)^{2g}$ and $x\in R^\gamma$, with puncture coordinates $z_i=g_iC_ig_i^{-1}$. Let $\omega_i\in W/W_{I_i}$ be such that $g_i\in B\omega_iL_iB$. Then \[F(\gamma,x) = (2g-1)d(t)+\sum\limits_{i=1}^{m-1}d_{\omega_i}(t).\]
\end{coro}
\begin{proof}
	As stabilizers are finite, $\gamma$ has finite order and, by \cite{Fogarty}, $R^\gamma$ is smooth and its tangent space at $x$ is $(T_xR)^\gamma$. Thus, $F(\gamma,x)$ is one half of the number of eigenvalues of $\gamma$ on $T_xR$ different from $1$, counted with multiplicity. This number is $2d(t)$ on $\g$ and $2d_{\omega_i}(t)$ on $\g/\operatorname{Ad}(g_{i})(\mathfrak{l}_{i})$.
\end{proof}

\section{Fixed points}

Given $\omega \in W$, let $\hat T^{\omega}$ be the preimage of $Z(G)$ under $t\mapsto {}^\omega tt^{-1}$. Let $(\overline{\omega},p)\in \mcI_G^{\mcC}$. We define
\[\tilde S_{\overline{\omega}}(p):=S(p)\bigcap_{i=1}^{g}\left(\hat T^{\omega_1^{(i)}}\cap\hat T^{\omega_2^{(i)}}\right)\]
and, extending the definition of Section \ref{sec:statement} (see Lemma \ref{lem:fixed-points}), $\tilde{Z}_{\overline{\omega}}(p)$ as the image of the map $\tilde S_{\overline{\omega}}(p)\to (T/Z(G))\times T^{2g}$ given in the first coordinate by the composition $\tilde S_{\overline{\omega}}(p)\subset T \to T/Z(G)$ and, in the last $2g$ coordinates, by $t\mapsto A(t)$, where $A:T\to T^{2g}$ is the morphism such that the action of $t$ on the torus components of $C_{\overline{\omega},p}$ is given by multiplication by $A(t)$.

\begin{lema}\label{lem:fixed-points}
	Let $(\overline{\omega},p)\in \mcI_G^{\mcC}$ with $p$ finishing at $e$. Then
	\begin{enumerate}
		\item $\tilde{Z}_{\overline{\omega}}(p)\simeq \tilde S_{\overline{\omega}}(p)/Z(G)$,
		\item $\tilde{Z}_{\overline{\omega}}(p)\subset (T/Z(G))\times Z(G)^{2g}$,
		\item $\gamma\in (T/Z(G))\times Z(G)^{2g}$ stabilizes some point of $C_{\overline{\omega},p}$ if and only if $\gamma\in \tilde{Z}_{\overline{\omega}}(p)$,
		\item if $T_{\overline{\omega},p}\to T$ is surjective, $\gamma\in (T/Z(G))\times Z(G)^{2g}$ stabilizes some point of $C_{\overline{\omega},p}(1)/U$ if and only if $\gamma\in \tilde{Z}_{\overline{\omega}}(p)$,
		\item for any $\gamma\in \tilde Z_{\overline{\omega}}(p)$, $[C_{\overline{\omega},p}^{\gamma}]= q^{-c_{\overline{\omega},p}(\gamma)} [C_{\overline{\omega},p}]$ for an integer $c_{\overline{\omega},p}(\gamma)$, and
		\item if $T_{\overline{\omega},p}\to T$ is surjective and $\gamma=(t,\zeta)\in \tilde Z_{\overline{\omega}}(p)$, $[C_{\overline{\omega},p}(1)^{\gamma}] = q^{d(t)-c_{\overline{\omega},p}(\gamma)}[C_{\overline{\omega},p}(1)]$.
	\end{enumerate}
\end{lema}
\begin{proof}
	The first two items are clear. Write $C_{\overline{\omega},p}$ as $T^{2g}  \times (\C^\times)^{S_p} \times \C^{N}$. If $\gamma$ fixes a point $(z,x,y)\in C_{\overline{\omega},p}$, then $\gamma\cdot  (z, x)=(z,x)$. We write $\gamma=(t,\zeta)$ with $t\in T/Z(G)$ and $\zeta\in Z(G)^{2g}$. Then $\gamma\cdot (z,x)=( \zeta^{-1}A(t) z, {}^ht\cdot x)$, where $h$ is the product of the commutators $[\omega_2^{(i)},\omega_1^{(i)}]$. It follows that $\zeta=A(t)$ and, thus, ${}^{\omega_i^{(j)}}tt^{-1}$ is central for every $i,j$. Hence, ${}^ht=t$ on $T/Z(G)$ and $t\in \tilde S_{\overline{\omega}}(p)$. This shows that $\gamma\in \tilde Z_{\overline{\omega}}(p)$ if it fixes some point. Conversely, every element of $\tilde Z_{\overline{\omega}}(p)$ fixes every point with $y=0$. Moreover, such an element $\gamma=(t,\zeta)$ fixes a point if and only if each non-zero coordinate of $y$ belongs to a root subgroup $U_\alpha$ for which $\alpha(t)=1$. This implies the last two claims. Item $(4)$ follows from Lemma \ref{lem:U-slice-trick} and item $(3)$.
\end{proof}

\begin{thm}\label{thm:codimension}
	Let $(\overline{\omega},p)\in \mcI_G^{\mcC}$ with $p$ finishing at $e$ and $\gamma=(t,\zeta)\in \tilde Z_{\overline{\omega}}(p)$. Then
	\[c_{\overline{\omega},p}(\gamma) = (2g+1)d(t) + \sum_{i=1}^{m-1}d_{\omega_i}(t).\]
\end{thm}
\begin{proof}
	To compute $c_{\overline{\omega},p}(\gamma)$, recall that the unipotent part of the cells is $T$-equivariantly
	\[U*U_{(\overline{\omega},p)} = U \times \left(\prod_{k\in D_p} U_{p_k^{-1}\cdot \alpha_k}\right) \times \left(\prod_{i=1}^g U_{\omega_1^{(i)}}^+ \times U_{\left(\omega_2^{(i)}\right)^{-1}}^+ \right)\times \left(\prod_{i=1}^{m-1}\left(U_{\omega_i^{-1}}^+\cap {}^{\omega_i}N_i\right)\right) \]
	and so, if $\Phi_i$ is the set of roots of $N_i$ and
	\[ \Phi_{(\overline{\omega},p)} = \Phi^+\cup \{p_k^{-1}\cdot \alpha_k:k\in D_p \} \cup\bigcup_{i=1}^g \left(\left(\Phi^+\setminus \Phi_{\omega_1^{(i)}}\right)\cup \left(\Phi^+\setminus \Phi_{\left(\omega_2^{(i)}\right)^{-1} }\right)\right)\cup\bigcup_{i=1}^{m-1} \omega_i(\Phi_i\setminus\Phi_{\omega_i}) \]
	as a multiset, then
	\[ c_{\overline{\omega},p}(\gamma) = |\{\alpha\in \Phi_{(\overline{\omega},p)}:\alpha(t)\neq 1\}|,\]
	where the elements are counted with multiplicity.

	We claim that
	\[|\{k\in D_p:(p_k^{-1}\cdot \alpha_k)(t) \neq 1\}| = \frac{1}{2} |\{k:(\pi_k^{-1}\cdot \alpha_k)(t) \neq 1\}|\]
	where on the right-hand side $k$ runs over all downs, stays and ups. Since $\gamma\in \tilde Z_{\overline{\omega}}(p)$, $t\in S(p)$ and thus $(\pi_k^{-1}\cdot \alpha_k)(t)=1$ for every stay $k$. Hence, $s_{\pi_k^{-1}\cdot \alpha_k}$ fixes $t$ for any stay and, by Lemma \ref{lem:symmetries}, $\pi_k^{-1}p_k$ fixes $t$ for any $k$. Therefore, $p_k^{-1}\cdot \alpha_k$ and $\pi_k^{-1}\cdot \alpha_k$ take the same value at $t$ for any $k$. This shows that
	\[|\{k\in D_p:(p_k^{-1}\cdot \alpha_k)(t) \neq 1\}| = |\{k\in S_p\cup D_p:(\pi_k^{-1}\cdot \alpha_k)(t) \neq 1\}|.\]
	Knowing this, the claim reduces to showing that
	\[\{p_k^{-1}\cdot \alpha_k: k\in D_p\} =  \{-p_k^{-1}\cdot \alpha_k: k\in U_p\}\]
	as multisets. Now, at ups, $\Phi_{p_k} \setminus \Phi_{p_{k-1}}= \{-p_k^{-1}\cdot \alpha_k\}$, while at downs, $\Phi_{p_{k-1}} \setminus \Phi_{p_{k}}= \{p_k^{-1}\cdot \alpha_k\}$. As the walk starts and finishes at $e$, each root is added at as many ups as it is removed at downs, and the desired equality holds.

	Now, the definition of $\tilde S_{\overline{\omega}}(p)$ shows that the elements ${}^{\omega_i^{(j)}}tt^{-1}$ are central for $j=1,\ldots, g$ and $i=1,2$. Thus, the subgroup generated by the $\omega_i^{(j)}$ preserves the roots $\alpha$  with $\alpha(t)\neq 1$. This observation, together with the previous analysis, implies that we can analyze each handle and puncture contribution separately.

	Let us deal with the subword $\omega \omega^{-1}$ associated with a puncture. In this case,
	\[\{\pi_k^{-1}\cdot \alpha_k \} = -\omega(\Phi_\omega) \cup -\Phi_{\omega^{-1}}\]
	and $\Phi_{\omega^{-1}} = -\omega(\Phi_\omega)$. Thus, $\frac{1}{2}|\{ k : (\pi_k^{-1}\cdot \alpha_k)(t)\neq 1\}|$ is nothing but the number of roots in $\omega (\Phi_\omega)$ that do not vanish at $t$. Hence, together with $\omega_i(\Phi_i\setminus\Phi_{\omega_i})$, each puncture contributes $d_{\omega_i}(t)$ to $c_{\overline{\omega},p}(\gamma)$.

	It remains to deal with a handle subword $\omega_1\omega_2\omega_1^{-1}\omega_2^{-1}$. In this case,
	\begin{align*}
		\{\pi_k^{-1}\cdot \alpha_k\} &= -\omega_2\omega_1\omega_2^{-1}(\Phi_{\omega_1})\cup -\omega_2\omega_1(\Phi_{\omega_2})\cup -\omega_2(\Phi_{\omega_1^{-1}}) \cup -\Phi_{\omega_2^{-1}} \\
		 &= -\omega_2\omega_1\omega_2^{-1}(\Phi_{\omega_1})\cup -\omega_2\omega_1(\Phi_{\omega_2})\cup \omega_2\omega_1(\Phi_{\omega_1}) \cup \omega_2(\Phi_{\omega_2}).
	\end{align*}
	Recall that $\omega_1$ and $\omega_2$ preserve the roots $\alpha$  with $\alpha(t)\neq 1$. Thus, together with $\Phi^+\setminus \Phi_{\omega_1}$ and  $\Phi^+\setminus \Phi_{\omega_2^{-1}}$, each handle contributes $2d(t)$.
\end{proof}

\begin{proof}[Proof of Theorem \ref{conjMotivic2}]
	Let $\gamma=(t,\zeta)\in\tilde Z_{\overline{\omega}}(p)$; by Lemma \ref{lem:fixed-points}(4), this is the set of Section \ref{sec:statement}. By Lemma \ref{lem:fixed-points}(6) and Theorem \ref{thm:codimension},
	\[[C_{\overline{\omega},p}(1)^\gamma]=q^{d(t)-c_{\overline{\omega},p}(\gamma)}[C_{\overline{\omega},p}(1)]\] for $c_{\overline{\omega},p}(\gamma)=(2g+1)d(t)+\sum\limits_{i=1}^{m-1}d_{\omega_i}(t)$. On the other hand, the cell $C_{\overline{\omega},p}$ lies in the piece of Theorem \ref{thm:adding-punctures} indexed by $(\omega_1,\ldots,\omega_{m-1})$, so its puncture coordinates are of the form $z_i=g_iC_ig_i^{-1}$ with $g_i\in B\omega_iL_iB$. The isomorphism $\Psi$ conjugates all the coordinates by an element of $U$, see the proof of Lemma \ref{lem:B-quotient}, so the same holds at the points of $Y$, where $\Psi(C_{\overline{\omega},p}(1))=Y\times U$. Hence, by Corollary \ref{cor:fermionic-shift}, the fermionic shift takes the value
	\[F(\gamma)=(2g-1)d(t)+\sum_{i=1}^{m-1}d_{\omega_i}(t)\]
	at all the points of $(C_{\overline{\omega},p}(1)/U)^\gamma\simeq Y^\gamma$. Thus, $F(\gamma)+d(t)=c_{\overline{\omega},p}(\gamma)-d(t)$, and the theorem follows from the reformulation at the end of Section \ref{sec:U-slice}.
\end{proof}

\section{Duality between walks}\label{sec:duality}

Let $G$ be as before (not necessarily simply connected) and $\pi:\tilde G\to G$ be its universal cover. We identify their Lie algebras by means of $d\pi$. This induces identifications between $\Delta_G^+$ and $\Delta_{\tilde G}^+$, between $W_G$ and $W_{\tilde G}$, and between walks on $G$ and on $\tilde G$. Given a walk $p$ on $G$, we denote by $\tilde p$ its associated walk on $\tilde G$.

Let ${}^L{G}$ be the connected and reductive algebraic group with root datum dual to that of $G$. The Weil pairing induces a group isomorphism between $W$ and $W^\vee$, the Weyl group of ${}^L{G}$. We can also identify each root of $ \Delta^+_G$ with its associated coroot in $\Delta^+_{{}^L{G}}$. Hence, $\Br$ gets identified with $\operatorname{Br}^+_{{}^L G}$. For each walk $p$ on $G$ we get an associated walk $p'$ on ${}^L{G}$, by staying at the same steps. Composing with $p\mapsto \tilde{p}$ we get a map $p\mapsto p^\vee:=\widetilde{p'}$. If $G$ is simply connected, this is a duality; $(p^\vee)^\vee=p$. We also write $\overline{\omega}\mapsto \overline{\omega}^\vee$ for the induced map $W^{m-1+2g}\to (W^\vee)^{m-1+2g}$. If $\mcC^\vee$ is dual to $\mcC$, this map identifies $\Omega_G^{\mcC}$ and $\Omega_{{}^L\tilde G}^{\mcC^\vee}$. Consequently, $(\overline{\omega},p)\mapsto(\overline{\omega}^\vee,p^\vee)$ is a bijection $\mcI_G^{\mcC}\to\mcI_{{}^L\tilde G}^{\mcC^\vee}$ preserving ups, downs, stays and the end of walks, as long as on ${}^L\tilde G$ we take as expression of each $\beta_{\overline\omega^\vee}$ the image of the one chosen for $\beta_{\overline\omega}$. Finally, we denote by $T^\vee\subset B^\vee$ the maximal torus and Borel subgroup of ${}^L\tilde G$ corresponding to $T\subset B$.

\begin{lema}\label{lem:tilde-dual}
	Let $\beta\in \Br$ and $p$ be an associated walk. Assume that $G$ is simply connected. Then $Z(p^\vee)=\pi_1(p)^*$ and $\pi_1(p^\vee)=Z(p)^*$.
\end{lema}
\begin{proof}
	By Lemma \ref{lem:Qvee}, $Q^\vee(p)=Q(p^\vee)$. In addition, since $G$ is simply connected, its coroot lattice $Q^\vee$ is $\t(\Z)$ by \cite[Section 4.3.5]{OV}. Hence,
	\[\pi_1(p)^* = ( Q^\vee/ Q^\vee(p) )^* = ( Q^\vee/ Q(p^\vee) )^* = Z(p^\vee). \]
	Similarly, $\pi_1(p^\vee)=Z(p)^*$.
\end{proof}

Before going on, let us spell out how we identify $Z({}^L\tilde G)^{2g}$ with $(Z(G)^{2g})^*=\operatorname{Hom}(Z(G)^{2g}, \C^\times)$ for $G$ simply connected. This amounts to a perfect pairing $Z(G)^{2g}\times Z({}^L\tilde G)^{2g}\to \C^\times$. Remember that the exponential induces isomorphisms $Z(G) \simeq P^\vee /Q^\vee$ and $Z({}^L\tilde G)\simeq P/Q$. Thus, the Weil pairing on $P\times P^\vee$ induces a perfect pairing $\langle\ ,\ \rangle: Z(G)\times Z({}^L\tilde G) \to \C^\times$. For $g\geq 0$, write elements of $Z(G)^{2g}$ as $(a,b)$ with $a,b\in Z(G)^g$ such that they act on $(x,y,z)\in R\subset G^g\times G^g\times \mcC_1\times\cdots\times\mcC_{m-1}$ by $(ax,by,z)$. The pairing we use then is $Z(G)^{2g}\times Z({}^L\tilde G)^{2g}\to \C^\times$, $((a,b),(c,d))\mapsto \prod\limits_{i=1}^g \langle a_i,d_i\rangle \langle b_i,c_i \rangle^{-1}$.

\begin{lema}\label{lem:tilde-dual-2}
	Let $(\overline{\omega},p)\in \mcI_G^{\mcC}$ with $p$ finishing at $e$. Assume that $G$ is simply connected. Then $\pi_{1,\overline{\omega}^\vee}(p^\vee) = \tilde Z_{\overline{\omega}}(p)^*$ and $\tilde Z_{\overline{\omega}^\vee}(p^\vee)=\pi_{1,\overline{\omega}}(p)^*$.
\end{lema}
\begin{proof}
	As in the proof of Theorem \ref{thm:cupping}(3), $\tilde Z_{\overline{\omega}}(p)\simeq(Q/\tilde Q_{\overline{\omega}}(p))^*$ with $\tilde Q_{\overline{\omega}}(p)=Q(p)+\sum_{i,j}((\omega_j^{(i)})^{-1}-\operatorname{Id})(Q)$. Now, the proof is analogous to that of Lemma \ref{lem:tilde-dual}.
\end{proof}

\begin{coro}\label{lem:duality-preserves-genericity}
	Let $(\overline{\omega},p)\in \mcI_G^{\mcC}$ with $p$ finishing at $e$. Assume that $G$ is simply connected and that $T_{\overline{\omega},p}\to T$ is surjective. Then $T_{\overline{\omega}^\vee,p^\vee}\to T^\vee$ is also surjective.
\end{coro}
\begin{proof}
	By Theorem \ref{thm:cupping}, $\pi_{1,\overline{\omega}}(p)$ is finite. Hence, by Lemma \ref{lem:tilde-dual-2}, $\tilde Z_{\overline{\omega}^\vee}(p^\vee)$ is finite too. Thus, $Z_{\overline{\omega}^\vee}(p^\vee)$ is finite and $T_{\overline{\omega}^\vee,p^\vee}\to T^\vee$ is surjective.
\end{proof}

As we have seen, the dual of the fundamental group is not $Z$ but rather an extension. This raises the question: what is the dual of $Z$?

\begin{Def}
	Let $(\overline{\omega},p)\in \mcI_G^{\mcC}$. We define
	\[\tilde Q^\vee_{\overline{\omega}}(p)=Q^\vee(p)+(\omega_1^{(1)}-\operatorname{Id})(P^\vee)+\ldots+(\omega_2^{(g)}-\operatorname{Id})(P^\vee),\]
	where $P^\vee=\{t\in\mathfrak t:\alpha(t)\in\mathbb Z\ \forall\alpha\in Q\}$, and
	\[\tilde\pi_{1,\overline{\omega}}(p):=\t(\Z)/\tilde Q^\vee_{\overline{\omega}}(p).\]
\end{Def}

\begin{lema}
	Let $(\overline{\omega},p)\in \mcI_G^{\mcC}$ with $p$ finishing at $e$ and $T_{\overline{\omega},p}\to T$ surjective, and $t\in T$. Then $\pi_0(C_{\overline{\omega},p}(t) / ((T/Z(G))\times Z(G)^{2g}))$ is $\tilde\pi_{1,\overline{\omega}}(p)$.
\end{lema}
\begin{proof}
	First of all, we note that $\pi_0(C_{\overline{\omega},p}(t) / ((T/Z(G))\times Z(G)^{2g})) = \pi_0(C_{\overline{\omega},p}(t))/  Z(G)^{2g}$ as $T$ is connected. We write $C_{\overline{\omega},p}$ as $T^{2g}\times (\C^\times)^{S_p}\times \C^N$. We recall that the map to $T$ was a composition of a translation, the projection onto $T^{2g}\times (\C^\times)^{S_p}$, and a group homomorphism $f':T^{2g}\times (\C^\times)^{S_p}\to T$, so that $\pi_0(C_{\overline{\omega},p}(t))\simeq \pi_0(\ker f')\simeq \ker f' /(\ker f')^\circ$, where $(\ker f')^\circ$ is the connected component of the identity in $\ker f'$. These maps are $(T/Z(G))\times Z(G)^{2g}$-equivariant.

	We recall that the action of $\zeta\in Z(G)^{2g}$ is given by translation by $\zeta^{-1}$ on $T^{2g}$. Hence, it acts trivially on $\pi_0(C_{\overline{\omega},p}(t))$ if and only if $(\zeta,1)\in (\ker f')^\circ$. Now, $(\ker f')^\circ$ agrees with the exponential of its tangent algebra while $\mcE^{-1}(Z(G))=P^\vee$.
	It follows that the group homomorphism $Z(G)^{2g} \to \ker f'/(\ker f')^\circ$, $\zeta\mapsto \zeta\cdot (\ker f')^\circ$, has image
	\begin{align*}
		(Z(G)^{2g}\times\{1\}) / ((Z(G)^{2g}\times\{1\})\cap (\ker f')^\circ ) &= \mcE((P^\vee)^{2g}\times \Z^{S_p})/\mcE(((P^\vee)^{2g}\times \Z^{S_p})\cap \ker(df'))\\
		&\simeq df'((P^\vee)^{2g}\times \Z^{S_p})/df'(\t(\Z)^{2g}\times \Z^{S_p})\\
		&= \tilde Q^\vee_{\overline{\omega}}(p)/Q^\vee_{\overline{\omega}}(p).
	\end{align*}
	Hence, under $\pi_0(C_{\overline{\omega},p}(t))\simeq \pi_{1,\overline{\omega}}(p) = \t(\Z)/Q^\vee_{\overline{\omega}}(p)$, the image corresponds to $\tilde Q^\vee_{\overline{\omega}}(p)/Q^\vee_{\overline{\omega}}(p)$. This shows that the quotient is $\tilde \pi_{1,\overline{\omega}}(p) = \t(\Z)/\tilde Q^\vee_{\overline{\omega}}(p)$.
\end{proof}

\begin{lema}
	Let $(\overline{\omega},p)\in \mcI_G^{\mcC}$ with $p$ finishing at $e$. Assume that $G$ is simply connected. Then $\tilde\pi_{1,\overline{\omega}^\vee}(p^\vee)=Z_{\overline{\omega}}(p)^*$ and $Z_{\overline{\omega}^\vee}(p^\vee)=\tilde\pi_{1,\overline{\omega}}(p)^*$.
\end{lema}
\begin{proof}
	The proof is analogous to that of Lemma \ref{lem:tilde-dual}.
\end{proof}

\begin{lema}\label{lem:key-duality}
	Let $(\overline{\omega},p)\in \mcI_G^{\mcC}$ with $p$ finishing at $e$ and $T_{\overline{\omega},p}\to T$ surjective. Assume that $G$ is simply connected. The dual of the projection $\tilde Z_{\overline{\omega}}(p)\to Z(G)^{2g}$ is the action map $Z({}^L\tilde G)^{2g}\to \pi_{1,\overline{\omega}^\vee}(p^\vee) = \pi_0(C_{\overline{\omega}^\vee,p^\vee}(1))$.
\end{lema}
\begin{proof}
	One must follow the identifications in Lemma \ref{lem:tilde-dual-2}.
\end{proof}

We are ready to prove Theorem \ref{conjMotivic}. First, let us compare cells associated to dual walks. For $(\overline{\omega},p)\in \mcI_G^{\mcC}$ with $p$ finishing at $e$ and $T_{\overline{\omega},p}\to T$ surjective, let $C_{\overline{\omega},p}^\circ$ be a connected component of $C_{\overline{\omega},p}(1)$ and $h_p:Z(G)^{2g}\to \pi_{1,\overline{\omega}}(p)=\pi_0(C_{\overline{\omega},p}(1))$ be the action map.

\begin{lema}\label{lem:connected-component-invariant}
	Let $(\overline{\omega},p)\in \mcI_G^{\mcC}$ with $p$ finishing at $e$ and $T_{\overline{\omega},p}\to T$ surjective, and assume that $G$ is simply connected. Then
	\begin{align*}
		[C_{\overline{\omega},p}^\circ / ((B/Z(G))\times (F^{2g}\cap \ker h_p))] &= [C_{\overline{\omega}^\vee,p^\vee}^\circ / ((B^\vee/Z({}^L\tilde G))\times (((Z(G)/F)^*)^{2g}\cap \ker h_{p^\vee}))]
	\end{align*}
	for any subgroup $F\subset Z(G)$.
\end{lema}
\begin{proof}
	Since the $U$-quotient is a vector bundle and $T/Z(G)$ is special,
	\[ [C_{\overline{\omega},p}^\circ / ((B/Z(G))\times (F^{2g}\cap \ker h_p))] [U] = [C_{\overline{\omega},p}^\circ / (F^{2g}\cap \ker h_p)][T/Z(G)]^{-1}. \]
	Now, $F^{2g}\cap \ker h_p$ acts by translation on the torus factor and hence
	\[[C_{\overline{\omega},p}^\circ/(F^{2g}\cap \ker h_p)] = [C_{\overline{\omega},p}^\circ],\]
	which is the class of a product of an affine space and a torus. Summing up, we get
	\[ [C_{\overline{\omega},p}^\circ / ((B/Z(G))\times (F^{2g}\cap \ker h_p))] = (q-1)^{j} q^{\frac{1}{2}(\dim \abbreviatedBetti[G]-j)}\]
	where $j = (2g-2)\dim T + |S_p|$, which is invariant under the duality.
\end{proof}

\begin{lema}\label{lem:16.4}
	Let $(\overline{\omega},p)\in \mcI_G^{\mcC}$ with $p$ finishing at $e$ and $T_{\overline{\omega},p}\to T$ surjective. Then
	\begin{align*}
		|h_p(F^{2g})| [(C_{\overline{\omega},p}(1)/U)^\gamma / ((T/Z(G))\times F^{2g})]  =  [C_{\overline{\omega},p}(1)^{[\gamma]}] [B]^{-1}
	\end{align*}
	for any subgroup $F\subset Z(G)$ and $\gamma\in \tilde Z_{\overline{\omega}}(p)\cap ((T/Z(G))\times F^{2g})$.
\end{lema}
\begin{proof}
	Since all connected components of $C_{\overline{\omega},p}(1)$ are isomorphic,
	\[	[(C_{\overline{\omega},p}(1)/U)^\gamma / ((T/Z(G))\times F^{2g})]  = \frac{| \pi_{1,\overline{\omega}}(p)|}{|h_p(F^{2g})|} [(C_{\overline{\omega},p}^\circ/U)^\gamma / ((T/Z(G))\times (F^{2g}\cap \ker h_p))].\]
	As in the previous lemma,
	\[[(C_{\overline{\omega},p}^\circ/U)^\gamma / ((T/Z(G))\times (F^{2g}\cap \ker h_p))] = [(C_{\overline{\omega},p}^\circ/U)^\gamma] [T/Z(G)]^{-1} \]
	and from this we get the desired equality, using that  the $U$-quotient is a vector bundle.
\end{proof}

\begin{proof}[Proof of Theorem \ref{conjMotivic}]
	By Corollary \ref{lem:duality-preserves-genericity} and Lemma \ref{lem:stringy-motive-and-cell-decomposition}, it suffices to show that
	\begin{multline*}
		\left[C_{\overline{\omega},p}(1)/((B/Z(G))\times F^{2g}), F_{\abbreviatedBetti[G/F]}\right]_{st} \\
		= \left[C_{\overline{\omega}^\vee,p^\vee}(1)/\left((B^\vee/Z({}^L\tilde G))\times \left((Z(G)/F)^*\right)^{2g}\right), F_{\mcM_{\mathrm{B}}^{\mcC^\vee}({}^L\tilde G/ (Z(G)/F)^*)}\right]_{st}
	\end{multline*}
	for any $(\overline{\omega},p)\in \mcI_G^{\mcC}$ with $p$ finishing at $e$ and $T_{\overline{\omega},p}\to T$ surjective. Now, by Theorem \ref{conjMotivic2} and Lemma \ref{lem:16.4}, the left-hand side is
	\[ |\tilde Z_{\overline{\omega}}(p)\cap ((T/Z(G))\times F^{2g})| \cdot [C_{\overline{\omega},p}(1) / ((B/Z(G))\times F^{2g})]\]
	where we cancel a factor $|h_p(F^{2g})|$ because all the classes involved lie in the subring generated by $q$, $q^{-1}$, and $(q-1)^{-1}$, which is a domain of characteristic zero.	We can rewrite this expression as
	\[ |\tilde Z_{\overline{\omega}}(p)\cap ((T/Z(G))\times F^{2g})| \cdot |\pi_{1,\overline{\omega}}(p) / h_p(F^{2g})| \cdot [C_{\overline{\omega},p}^\circ / ((B/Z(G))\times (F^{2g}\cap \ker h_p))]. \]
	There is also a similar formula for the dual group. Now, by Lemma \ref{lem:key-duality}, the product
	\[|\tilde Z_{\overline{\omega}}(p)\cap ((T/Z(G))\times F^{2g})| \cdot |\pi_{1,\overline{\omega}}(p) / h_p(F^{2g})|\]
	is preserved under duality, while the other factor is invariant by Lemma \ref{lem:connected-component-invariant}.
\end{proof}

\printbibliography[heading=toplevel]

\end{document}